\documentclass{article}

\usepackage[english]{babel}

\usepackage[letterpaper,top=2cm,bottom=2cm,left=3cm,right=3cm,marginparwidth=1.75cm]{geometry}

\usepackage{graphicx}
\usepackage[colorlinks=true, allcolors=blue]{hyperref}
\usepackage{stmaryrd}
\usepackage[T1]{fontenc}
\DeclareUnicodeCharacter{2010}{-}
\usepackage[utf8]{inputenc}
\usepackage[english]{babel}
\usepackage{geometry}
\usepackage{graphicx}
\usepackage{mathtools}
\usepackage{listings}
\usepackage[printonlyused,withpage]{acronym}
\usepackage{siunitx}
\usepackage{multirow}
\usepackage{setspace}
\usepackage{eufrak}
\usepackage[dvipsnames]{xcolor}
\usepackage{listings}
\usepackage{tabularx}
\usepackage{marvosym}
\usepackage{amsmath, amssymb, amsthm, amsfonts}
\usepackage{amssymb}
\usepackage{subcaption}
\usepackage{physics}
\usepackage{algorithm}
\usepackage{algpseudocode}
\usepackage{lipsum}
\usepackage[usestackEOL]{stackengine}
\usepackage{etoolbox}
\mathtoolsset{showonlyrefs}
\usepackage[mathscr]{euscript}
\theoremstyle{plain} 
\newtheorem{theorem}{Theorem}
\numberwithin{theorem}{section}

\newtheorem{lemma}[theorem]{Lemma}

\theoremstyle{definition}

\newtheorem{example}[theorem]{Example}

\newtheorem{remark}[theorem]{Remark}
\usepackage[printonlyused,withpage]{acronym}
\usepackage{multicol}
\usepackage{pst-node}
\usepackage{tikz-cd}
\usepackage{physics}
\usepackage{algorithm}
\usepackage{algpseudocode}
\usepackage{enumerate}
\usepackage{svg}
\usepackage{accents}
\usepackage{indentfirst}

\newcommand{\NNN}{\nonumber\\}

\usepackage{fullpage,graphicx,tikz}
\usepackage{wrapfig}
\usepackage{mathtools,graphicx}
\usetikzlibrary{tikzmark}
\usetikzlibrary{shapes,arrows}
\usetikzlibrary {decorations.markings}
\usetikzlibrary {decorations.pathreplacing}

\newlength{\fromfoliotocopyright}
\providecommand{\keywords}[1]
{
  \small	
  \textbf{\textit{Keywords---}} #1
}
\providecommand{\classification}[1]
{
  \small	
  \textbf{{AMS subject classifications--}} #1
}
\title{Offline-online approximation of multiscale eigenvalue problems with random defects}
\author{Dilini Kolombage\footnotemark[1] \and Barbara Verf{\"u}rth\footnotemark[1]}
\date{}

\begin{document}
\maketitle
\renewcommand{\thefootnote}{\fnsymbol{footnote}}
\footnotetext[1]{Institut für Numerische Simulation, Universität Bonn, Friedrich-Hirzebruch-Allee 7, D-53115 Bonn, Germany}
\renewcommand{\thefootnote}{\arabic{footnote}}

\begin{abstract}
In this paper, we consider an elliptic eigenvalue problem with multiscale, randomly perturbed coefficients. For an efficient and accurate approximation of the solutions for many different realizations of the coefficient, we propose a computational multiscale method in the spirit of the Localized Orthogonal Decomposition (LOD) method  together with an offline-online strategy similar to [M{\aa}lqvist, Verf\"urth, \emph{ESAIM Math. Model. Numer. Anal.}, 56(1):237–260, 2022]. The offline phase computes and stores local contributions to the LOD stiffness matrix for selected defect configurations. Given any perturbed coefficient, the online phase combines the pre-computed quantities in an efficient manner. We further propose a modification in the online phase, for which numerical results indicate enhanced performances for moderate and high defect probabilities. We show rigorous a priori error estimates for eigenfunctions as well as eigenvalues.
\end{abstract}

\keywords{eigenvalue problems, multiscale method, perturbed coefficients, offline-online strategy}\\

\classification{65N25, 65N30, 65N12, 65N15, 35J15}

\section{Introduction}
\label{sec:Intro} Metamaterials are modern, multiscale materials and have become increasingly popular for their unique mechanical, acoustic, or electromagnetic properties. Structurally, these materials are often constructed with a uniform distribution of very small heterogeneities, i.e., they can be represented as a periodic arrangement of unit cells. However, manufacturing mistakes that occur during the fabrication process introduce random flaws into the material structure. Therefore, the robustness study of material properties in the presence of random defects is a particularly interesting research topic in modern material sciences. From a mathematical point of view, this is modeled by an elliptic partial differential equation with a randomly perturbed coefficient. Such a stochastic partial differential equation can be solved in several ways; for instance,  by stochastic Galerkin or stochastic collocation methods, see \cite{MR2084236, MR1870425, MR3202242, MR2318799} for an overview. Another large class of methods consists of sampling-based approaches, in particular Monte Carlo (MC)-type methods, including multilevel and quasi MC methods. We refer to \cite{MR2824857, MR2835612,MR3024159} for an overview.

In this work, we aim to obtain accurate and efficient approximations of the solutions to the stochastic PDE for many different samples. However, in a multiscale setting, standard finite element methods are, in general, not feasible due to the required fine discretization. Therefore, computational multiscale methods are employed, where we focus on the Localized Orthogonal Decomposition (LOD) method \cite{MR4298217,MR4191211, MR3240855, MR3926249, MR3246801} in this work.  The key idea is to construct problem-adapted basis functions incorporating information about the multiscale nature of the problem, so that a low-dimensional ansatz space yields sufficiently accurate solutions. However, solving the problem for many different samples remains quite expensive, as the problem-adapted basis functions must be re-computed for each realization of the (perturbed) coefficient.

The combined challenges of randomness and multiscale features have been tackled in several studies, often in the MC context and based on different multiscale methods. In \cite{MR3605827}, the authors construct so-called gamblets -- hierarchical basis functions that enable sparse and orthogonal multiresolution decomposition. The work in \cite{MR3118248} considers multilevel MC methods for problems with several, separated scales. In \cite{MR3649314}, a finite element-based hierarchical MC method is presented for problems with periodic multiscale coefficients generated by an ergodic dynamical system acting on a probability space.
Multiscale model reduction is combined with data-driven approaches in \cite{MR3301304} to obtain a reduced dimensional system in both physical and stochastic aspects. In \cite{MR3264336}, the authors propose a Multiscale Finite Element Method with a deterministic multiscale basis by taking into account the perturbative nature of the random coefficients, which have a similar structure as our model below. Very recently, \cite{elasmi2024neuralnumericalhomogenizationbased} suggests a combination of the LOD with neural networks for an efficient evaluation of the coefficient-to-basis map for general parametrized multiscale problems, that might be applicable in the random case as well. All the aforementioned works primarily consider the source problem. With regard to the eigenvalue problem, \cite{MR4027848} combines operator-adapted wavelets (gamblets) with hierarchical subspace correction to efficiently block-diagonalize the operator and refine solutions across scales. The algorithm computes eigenpairs through coarse-to-fine corrections in a multigrid fashion. A preconditioned inverse iteration with an optimal multigrid solver is used in \cite{MR4028792} to efficiently compute approximations of the localized low-energy eigenstates to a linear random Schr\"odinger operator.
Note that we cannot expect such localized eigenstates in our setting.

Here, we study the elliptic eigenvalue problem with perturbed coefficients
 \begin{equation}
 \label{elliptic:strong}
      -\nabla \cdot (A(x,\omega)\nabla u(x,\omega))=\lambda(\omega)u(x,\omega)
 \end{equation}
 
where $A(x,\omega)$ is a realization of the form
\begin{equation}
\label{coefA}
    A(x, \omega) = A_\varepsilon(x) + b_p(x;\omega)B_\varepsilon(x).
\end{equation}
 
Here, $A_\varepsilon(x) $ and $B_\varepsilon(x)$ are deterministic, multiscale coefficients, and $b_p(x;\omega)$ follows a Bernoulli law with probability $p$. Detailed assumptions are postponed to Section \ref{sec:pre_1}. The random defects model for $A$ is the same as in \cite{MR4378546} and originally considered in \cite{MR2818410, MR2864078}. Figure \ref{fig:image1} represents two example coefficients of $A$  which fall into the above model. In the random checkerboard (Figure \ref{fig:1.1}), the coefficient $A$ is generated by randomly including $\varepsilon$-length square-spots, \textit{i.e. defects}, with a probability $p$ on to the constant background coefficient $A_\varepsilon$. In random erasure (Figure \ref{fig:1.2}), $A_\varepsilon$ is made of square-spots that are periodically arranged with periodicity $\varepsilon$ on a constant background value. Then the coefficient $A$ is generated by randomly removing the $\varepsilon$-length square-spots with a probability $p$. While in the random checkerboard a defect is represented by the random inclusion of a square-spot, in the random erasure a defect is indicated by the random removal of a square-spot.
    
\begin{figure}[H]
\begin{subfigure}{0.5\textwidth}
\includegraphics[width=0.8\linewidth, height=5cm]{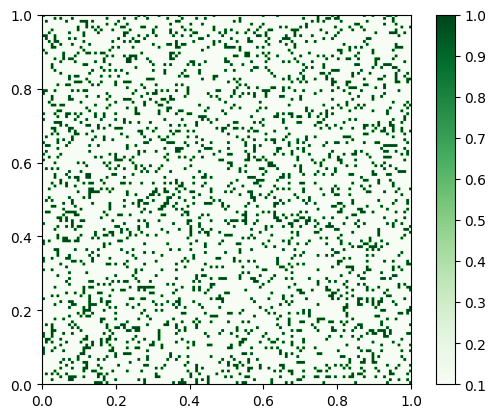} 
\caption{Random checkerboard}
\label{fig:1.1}
\end{subfigure}
\begin{subfigure}{0.5\textwidth}
\includegraphics[width=0.8\linewidth, height=5cm]{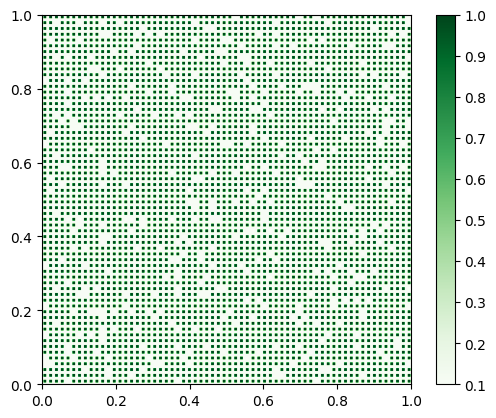}
\caption{Random erasure}
\label{fig:1.2}
\end{subfigure}
\caption{Weakly random coefficients with $\varepsilon = 2^{-7}$ and $p=0.1$.}
\label{fig:image1}
\end{figure}

Our work combines the LOD for multiscale eigenvalue problems, but without random perturbations, from \cite{MR3343928} and the offline-online strategy for randomly perturbed source-type diffusion problems from \cite{MR4378546}. The latter allows for a fast and efficient approximation of the LOD stiffness matrix for many realizations. To transfer the idea to eigenvalue problems, we employ the coefficient-independent finite element mass matrix for the right-hand side instead of the LOD mass matrix considered in \cite{MR3343928}. Moreover, we propose a modification to the offline-online strategy improving the approximation in particular for moderate to high defect probabilities, which may be of independent interest.
Besides the formulation of this method, our main contribution is a detailed a priori error analysis, based on perturbation arguments and Babu\v{s}ka-Osborn theory \cite{MR1115240, MR962210} for eigenvalue problems. As intermediate results that may be of own interest, we derive eigenfunction and eigenvalue error estimates for the Petrov-Galerkin variant of the LOD and with finite element mass matrix. We emphasize that, in contrast to \cite{MR3343928}, we always consider the \emph{localized} version of the LOD method, where the corrector problems for the multiscale basis functions are computed on local patches as done in practice. 
  
The paper is organized as follows. Section \ref{sec:2} introduces the model problem and recaps the offline-online strategy for the LOD method, applied to the eigenvalue problem. Further, we propose a new modification of the offline phase. In Section \ref{sec:3}, we present our preliminary work on compact operators for eigenvalue problems that has been used as the foundation for the error analysis in Section \ref{sec:4}. In Section \ref{sec:5}, several numerical experiments illustrate the theoretical results and assess the performance of the proposed method. The conclusion is given in Section \ref{sec:6}.

\section{Setting and numerical methods}
\label{sec:2}

The main goal of this section is to introduce the offline-online strategy inspired by the LOD method and apply it to our eigenvalue problem. For this, we first introduce our model problem with all necessary assumptions. To derive our method, we recap the essential ingredients of the LOD. Then we discuss how to discretize the mass term on the right-hand side in the eigenvalue problem, which leads to the so-called modified LOD (MLOD). Finally, we recall the offline-online strategy  -- here called OLOD -- based on \cite{MR4378546}, which allows for an efficient assembly of the left-hand side of the eigenvalue problem. We also propose an alternate variant of the OLOD at the end of the section.

\subsection{Model problem}
\label{sec:pre_1}
In this section, we describe the problem for a single sample $\omega$. For simplicity, we thus omit the random variable $\omega$ in the following. Recall, however, that the final goal is to find an effective way of solving this problem for many samples. We assume the domain $\Omega \subset \mathbb{R}^d$ to be a unit hypercube, i.e., $\Omega = [0,1]^d$ with $d\in\{1,2,3\}$ where we impose periodic boundary conditions. When we later discretize our problem, we always use Cartesian meshes. 

\noindent The weak form of the elliptic eigenvalue problem in \eqref{elliptic:strong} is given by: Find $(\lambda, u) \in \mathbb{R} \times V$ such that
    \begin{equation}
    \label{eqnn: 2.7}
        a(u,v) = \lambda m(u,v), \; \forall v \in V,
    \end{equation}
\noindent where \begin{equation*}
    a(u, v) := \int\limits_{\Omega}A(x)\nabla u(x)\cdot \nabla v(x) dx, \hspace{2cm} m(u, v) := \int\limits_{\Omega} u(x)v(x) dx.
\end{equation*} 
Note that although equation \eqref{eqnn: 2.7} is stated in a general form for simplicity, it encloses all eigenpairs $(\lambda^{(l+j)},u^{(l+j)})$ for $j=0, \cdots ,r-1$, accounting for multiplicities of the $l$-th eigenvalue. This interpretation applies similarly to all subsequent discretized formulations.
The functional space $V$ is the conforming finite element space given by
\begin{equation*}
    V := V_h \subseteq \mathcal H^1_\#(\Omega) := \left\{ v \in H^1_{\text{loc}}\left(\mathbb{R}^d\right) \ \middle|\ v \text{ is } \Omega\text{-periodic and } v|_\Omega \in H^1(\Omega) \right\},
\end{equation*}
where $ H^1(\Omega) := \{ v \in L^2(\Omega) : \nabla v \in (L^2(\Omega))^d \} $ is the usual Sobolev space, and $H^1_{\text{loc}}(\mathbb{R}^d)$ consists of functions in $ H^1(\mathcal{U}) $ for every bounded open set $ \mathcal{U} \subset \mathbb{R}^d $. The \(\Omega\)-periodicity means that the function $ v$ can be extended to a function on $\mathbb{R}^d$ that is periodic with respect to the domain $\Omega$.
 We assume the mesh size $h$ to be small enough that we treat the solution $u\in V_h$ as the reference solution, i.e., the error towards the exact solution in $\mathcal{H}^1_\#(\Omega)$ is negligible.
\noindent The energy norm and the $L^2$-norm are as usually defined by $a(\cdot, \cdot) =: {\norm{\cdot}^2_{\mathcal{A}}}={\norm{A^{1/2}\nabla \cdot}^2_{L^2(\Omega)}}$  and $m(\cdot, \cdot) =: \norm{\cdot}^2_{L^2(\Omega)}$ respectively. 
The realization $A(x) \in L^\infty(\Omega, \mathbb{R})$ is a sample of a randomly perturbed coefficient which is uniformly bounded and elliptic
\begin{equation*}
        0<\alpha \leq \text{ess} \inf\limits_{x\in \Omega}A(x), \;\;\;\; \infty > \beta \geq \text{ess}\sup\limits_{x\in \Omega}A(x).
\end{equation*}
For simplicity, we assume in the following that $A(x)$ is scalar-valued; however, all results can be immediately transferred to the matrix-valued case. 
The underlying deterministic coefficients $A_\varepsilon$ and $B_\varepsilon$ of $A(x)$ given in \eqref{coefA} are taken to be spectrally bounded as 
\begin{equation*}
    0<\alpha \leq \text{ess} \inf\limits_{x\in \Omega}A_{\text{per}}(x), \;\;\;\; \infty > \beta \geq \text{ess}\sup\limits_{x\in \Omega}A_{\text{per}}(x)
\end{equation*}
and 
  \begin{equation*}
    0<\alpha \leq \text{ess} \inf\limits_{x\in \Omega}(A_{\text{per}}(x)+B_{\text{per}}(x)), \;\;\;\; \infty > \beta \geq \text{ess}\sup\limits_{x\in \Omega}(A_{\text{per}}(x)+B_{\text{per}}(x))
\end{equation*}  
for  $A_{\text{per}}$ and $B_{\text{per}}$ are  $1-$periodic with $A_\varepsilon=A_{\text{per}}(x/\varepsilon)$ and $B_\varepsilon = B_{\text{per}}(x/\varepsilon)$. The multiscale parameter $\varepsilon$ is such that $\varepsilon=1/n, \; n \in \mathbb{N}, n \gg 1$. The random defects are modeled via
    \begin{equation*}
       b_{p,\varepsilon}(x,\omega) = \sum\limits_{i \in I}\chi_{\varepsilon(i+Q)}(x)\hat{b}^i_p(\omega).
    \end{equation*}
    
\noindent Here, the function $\chi$ is the characteristic function of each defect with $Q \subseteq [0,1]^d$ and $I:= \{\kappa \in \mathbb{Z}^d\ | \varepsilon(\kappa+Q)\subset \Omega\}$. The independent random variables $\hat{b}^i_p$ are Bernoulli distributed with probability $p$. Precisely, $\hat{b}^i_p=0$ with probability $1-p$ and $\hat{b}^i_p=1$ with probability $p$. 

The analysis of the elliptic eigenvalue problem guarantees the existence of its eigensolutions \cite{MR2597943}. The well-posedness of the problem \eqref{eqnn: 2.7} is guaranteed by the coercivity and the continuity of the bilinear form $a(\cdot, \cdot)$ \cite{MR443377}.

For conciseness, we use the notation $a \lesssim b$ to denote the relationship $a \leq Cb$ where the generic constant $C$  is independent of the discretization parameters $H, h, \mathfrak{k}$ as well as $\varepsilon$ throughout the rest of the paper.

\subsection{The LOD method}
The LOD method is a numerical approach designed for solving elliptic differential equations with highly heterogeneous or multiscale coefficients. It enables us to accurately approximate solutions without requiring a fine mesh over the whole domain. Instead, it works on a coarse mesh while incorporating localized fine-scale corrections. In this section, let us briefly outline the LOD method.

We start by introducing a coarse mesh $\mathcal{T}_H$ over the domain $\Omega$, with a local mesh size $H  := \max\limits_{T\in \mathcal{T}_H}\text{diam}(T)$. The fine-scale behavior is assumed to occur on a smaller scale $\varepsilon \ll H$. We define the computational finite element space $V_H$, consisting of continuous, piecewise bilinear functions on quadrilateral (or cuboidal) elements by
\begin{equation*}
    V_H := \mathcal{H}^1_\#(\Omega) \cap \mathcal{Q}_1(\mathcal{T}_H)
\end{equation*}
where $\mathcal{Q}_1(\mathcal{T}_H)$ is the space of piecewise polynomials of partial degree $\leq 1$.

The fine space $V_h$ on which the reference solution is defined, is constructed similarly as $V_H$ but based on a refined mesh $\mathcal T_h$. We assume that the  coarse mesh size $H$ is an integer multiple of the periodicity length $\varepsilon$, which together with the periodicity of the coefficient gives additional structure to our problem.

To distinguish coarse- and fine-scale features, the LOD method utilizes a quasi-local linear operator $I_H: V \to V_H$. The kernel space $$W := \text{ker}(I_H\vert_{V})=\{v \in V : I_H(v) = 0\} \subset V$$ captures the fine-scale components unresolved by $V_H$.
\newline

\noindent In this study, we define this bounded, surjective projection operator $I_H$ as follows:
\begin{equation}
\label{def: I}
    I_H := E_H \circ \Pi_H,
\end{equation}
where $\Pi_H : V \to \mathcal{Q}_1(\mathcal{T}_H)$ denotes the local $L^2$-projection onto the piecewise linear discontinuous finite element space, and $E_H$ is the averaging operator that maps discontinuous functions in $\mathcal{Q}_1(\mathcal{T}_H)$ to the conforming space $V_H$ defined by 
\begin{equation*}
    (E_H(v))(z) = {\sum\limits_{T\in \mathcal{T}_H, z \in \overline{T}} v|_T(z)}\bigg/{\text{card}\{K \in \mathcal{T}_H, z\in \overline{K}\}},
\end{equation*}
for any $v \in \mathcal{Q}_1(\mathcal{T}_H)$ and vertex $z$ of $\mathcal{T}_H$. This interpolation operator $ I_H = E_H \circ \Pi_H$ is $L^2$-stable as discussed in \cite{MR3687898}, that is,
\begin{equation*}
    \norm{I_Hv}_{L^2(\Omega)} \leq C_{\text{stab}}\norm{v}_{L^2(\Omega)} \text{ for all } v\in V
\end{equation*}
with a constant $C_{\text{stab}}$ that depends only on the mesh regularity.

Having established the construction and stability of the interpolation operator $I_H$, we now turn to the construction of the LOD method. In the ideal setting, fine-scale corrections are computed globally by a corrector operator  $  \mathscr{C}(A): V \to W $ defined over the entire domain. This operators is defined via its $a$-orthogonal property \begin{equation*}
 a((1-\mathscr{C}(A))v, w) =0, \; \; \text{ for } v \in V \text{ and }w \in W.
\end{equation*} 
The key idea is to use the finite-dimensional space $(1-\mathscr C(A))V_H$ in a Galerkin scheme for \eqref{eqnn: 2.7}. Although computing global correctors is computationally expensive, their exponential decay from the associated coarse elements enables efficient local approximation within small patches of the domain. We introduce these patches as follows.

For an element $T \in \mathcal{T}_H$ and $ \mathfrak{k} \in \mathbb{N}_0$, the patch $U_\mathfrak{k}(T) \subset \Omega$ centered around $T$ is defined by
\begin{equation*}
    \begin{split}
        U_0(T) & = T \\
        U_{\mathfrak{k}+1}(T) &= \bigcup\{T\in \mathcal{T}_H: \; \overline{T}\cap \overline{U_\mathfrak{k}(T)} \neq 0\}.
    \end{split}
\end{equation*}
The patch $U_\mathfrak{k}(T)$ is called the $\mathfrak{k}$-th layer patch. The localized fine-scale space associated with a patch $U_\mathfrak{k}(T)$ is then defined as
\begin{equation*}
       W^\mathfrak{k} :=  W(U_\mathfrak{k}(T)) := \{ w \in W : \; \text{supp}(w)\subseteq U_\mathfrak{k}(T)\}.
\end{equation*}

The localized corrector problem defined on $W^\mathfrak{k}$ is designed to compute a fine-scale correction that removes unresolved oscillations from any coarse function. In the following, we define the element correction operator, $\mathscr{C}_{\mathfrak{k},T}(A): V \to W^\mathfrak{k}$ that solves
\begin{equation*}
\label{eqn: 6.23}
  a\left(\mathscr{C}_{\mathfrak{k},T}(A) v, w\right) =  \left(A\nabla(\mathscr{C}_{\mathfrak{k},T}(A)v),\nabla w\right)_{U_\mathfrak{k}(T)} = \left(A\nabla v, \nabla w\right)_T = \int\limits_T(A\nabla v)\cdot \nabla w dx, \; \; \forall w \in W^\mathfrak{k}.
\end{equation*}
Observe that the space $W^\mathfrak{k}$ is constructed by restricting the global fine-scale space $W$ to the localized patches where $I_H$ is our continuous quasi-interpolation operator. Since the kernel of a continuous linear operator on a Hilbert space is closed, and the restriction preserves closedness, it follows that $W^\mathfrak{k} \subset \mathcal{H}^1_\#(\Omega)$ is also closed. Consequently, the element correction operator defined through a local linear variational problem on $W^\mathfrak{k}$ is well-defined by the Lax-Milgram theorem. \newline 

\noindent Next we define the localized correction operator $\mathscr{C}_\mathfrak{k}(A): V \to  W^\mathfrak{k}$ as the sum
\begin{equation*}
    \mathscr{C}_\mathfrak{k}(A) := \sum\limits_{T\in \mathcal{T}_H}\mathscr{C}_{\mathfrak{k},T}(A).
\end{equation*}

\begin{remark}
We stress stress that the localized correction operator $\mathscr{C}_\mathfrak{k}$ does not inherit the $a$-orthogonality property of the ideal correction operator $\mathscr{C}(A)$ due to the truncation errors. In addition to being $a$-orthogonal, the ideal corrector $\mathscr{C}(A)$ is also stable in the energy norm for all $v\in V$, as it acts as an $a$-orthogonal projection onto the fine-scale space $W$. Since the orthogonal projections in Hilbert spaces do not increase the norm, $\norm{\mathscr{C}(A)v}_\mathcal{A}$ is guaranteed to be bounded by $\norm{v}_\mathcal{A}$. Although the localized corrector $\mathscr{C}_\mathfrak{k}(A)$ lacks the $a$-orthogonality, it is stable in the energy norm for any $v \in V_H$. This follows because the element correction operator is well defined and $\mathscr{C}_{\mathfrak{k},T}(A)v$ solves a coercive problem on a patch where these patches have uniformly bounded overlap.
\end{remark}

\noindent We now define the localized multiscale space by
\begin{equation}
\label{eqn: 16}
    V^\mathrm{ms}_{H,\mathfrak{k}} := \left(1 - \mathscr{C}_\mathfrak{k}(A)\right) V_H = \text{span}\{\varphi_z-\mathscr{C}_\mathfrak{k}(A)\varphi_z\}_{z \in N}
\end{equation}
where $N$ denotes the set of vertices of $\mathcal{T}_H$ and $\{\varphi_z\}_{z\in N}$ is the nodal basis of $V_H$. \newline

\noindent From this point onward, we abbreviate $\mathscr{C}(A) = \mathscr{C}, \mathscr{C}_\mathfrak{k}(A) = \mathscr{C}_\mathfrak{k}$ and $  \mathscr{C}_{\mathfrak{k},T}(A)= \mathscr{C}_{\mathfrak{k},T}$. We can now introduce the LOD problem.\\

\noindent \textbf{The LOD problem:} 
Find $({\lambda}_{H,\mathfrak{k}},{u}_{H,\mathfrak{k}})\in \mathbb{R} \times V_{H}$ such that

\begin{equation}
\label{eqn: LOD}
{a}_\mathfrak{k}({u}_{H,\mathfrak{k}}, v) = {\lambda}_{H,\mathfrak{k}}{m}_\mathfrak{k}({u}_{H,\mathfrak{k}}, v) \;\; \forall {v} \in  V_{H}
\end{equation}
where ${a}_\mathfrak{k}(\cdot, \cdot) = {a}((1-\mathscr{C}_\mathfrak{k})(\cdot), (1-\mathscr{C}_\mathfrak{k})(\cdot))$ and ${m}_\mathfrak{k}(\cdot, \cdot)={m}((1-\mathscr{C}_\mathfrak{k})(\cdot), (1-\mathscr{C}_\mathfrak{k})(\cdot))$.\newline

\noindent For a comprehensive treatment of the LOD method, we refer the reader to \cite{MR4191211}, and include the following results \cite[Theorem 4.3, Lemma 8.1]{MR4191211} without proof.

\begin{theorem}
\label{thm:5.2.4}
For any $v\in V$ and any $\mathfrak{k} \in \mathbb{N}_0$
\begin{equation*}
    \norm{(\mathscr{C}- \mathscr{C}_\mathfrak{k})v}_{\mathcal{A}} \lesssim g(\mathfrak{k})\norm{v}_{\mathcal{A}},
\end{equation*}
where $g(\mathfrak{k}) = (\mathfrak{k}^{d/2}+1)\exp(-c\frac{\alpha}{\beta}\mathfrak{k})$.
\end{theorem}

\begin{lemma}
\label{Lem: 5.1.4}
Let $I_H=E_H \circ \Pi_H$ be as defined in \eqref{def: I}. Then, for $v \in V$ and $w \in W = \mathrm{kern}(I_H)$,
\begin{equation*}
    \frac{\abs{m(v,w)_{L^2(\Omega)}}}{{\norm{v}_{{\mathcal{A}}}}{\norm{w}_{{\mathcal{A}}}}} \lesssim H^2.
\end{equation*}
\end{lemma}

\noindent Without localization, i.e., $ \mathfrak{k}=\infty$, the ideal LOD method for the problem \eqref{eqn: LOD}, has been proposed and rigorously analyzed in \cite{MR3343928, MR4191211}. Combining these estimates and proof techniques with Theorem 2.2, one arrives at the following error estimates for the $l$-th eigenvalue solution of the LOD with localization

\begin{equation}
    \begin{split}
    \label{est:LOD}
        \frac{\lambda^{(l)}_{H,\mathfrak{k}}-\lambda^{(l)}}{\lambda^l} & \lesssim \left( g(\mathfrak{k}) + H^2\right) H^2  \\
     \norm{u^{(l)}-{u}_{H,\mathfrak{k}}^{(l),\mathrm{ms}}}_{L^2(\Omega)}
     & \lesssim   \left(g(\mathfrak{k}) + H^2\right)H \\ 
      \norm{u^{(l)} - {u}_{H,\mathfrak{k}}^{(l),\mathrm{ms}}}_\mathcal{A} & \lesssim \left( \sqrt{g(\mathfrak{k})}+H \right)H
    \end{split}
\end{equation}

\noindent where $u^{(l),\text{ms}}_{H,\mathfrak{k}}=(1-\mathscr{C}_\mathfrak{k})u^{(l)}_{H,\mathfrak{k}}$.
By choosing $\mathfrak{k}$ large enough such that $g(\mathfrak{k}) \approx H^2$, we obtain the same order of convergence in \eqref{est:LOD} as established in \cite[Theorems 8.4, 8.5]{MR4191211}.

\subsection{Modified LOD (MLOD) method}
In the MLOD method, we consider a Petrov-Galerkin version. While the trial space is taken to be the multiscale space $V^\mathrm{ms}_{H,\mathfrak{k}}$, the test space is considered to be the standard FE space $V_H$. This is advantageous in terms of computational complexity, see \cite{MR3422449} for details, and, in fact, the Petrov-Galerkin ansatz is essential for the offline-online strategy in the next section. Furthermore, we use the standard FE mass matrix instead of the LOD mass matrix. This choice is justified since the right-hand side does not involve the multiscale coefficient and thus does not need to be computed on a multiscale space. Choosing the FEM mass matrix is advantageous for our offline-online strategy as it allows to compute the right-hand side independent from the random sample. We will carefully analyze the error introduced by these modifications in Section \ref{sec: 4.1} and \ref{sec: 4.2}.\\

\noindent\textbf{The MLOD problem:} 
 Find $(\Tilde{\lambda}_{{H,\mathfrak{k}}},\Tilde{u}_{{H,\mathfrak{k}}})\in  \mathbb{R} \times V_{H} $ such that
\begin{equation}
\label{eqnn: 5.399}
\widetilde{a}_\mathfrak{k}(\Tilde{u}_{{H,}\mathfrak{k}},v) = \Tilde{\lambda}_{{H,}\mathfrak{k}} \widetilde{m}_\mathfrak{k}(\Tilde{u}_{{H,}\mathfrak{k}}, v), \;\;\; \forall v \in  V_{H}
\end{equation}
where $\widetilde{a}_\mathfrak{k}(\cdot,\cdot) = a((1-\mathscr{C}_\mathfrak{k})(\cdot), (\cdot)),$ and $ \widetilde{m}_\mathfrak{k}(\cdot, \cdot) = m(\cdot, \cdot)$.

\subsection{The offline-online strategy}
\label{off_on}
The equation \eqref{eqn: 16} tells us that for each realization of $A$, we have to re-compute the correction operator $\mathscr{C}_\mathfrak{k}$, and consequently the LOD basis functions. This may be computationally very demanding although the computation of $\mathscr{C}_\mathfrak{k}$ can be parallelized. As a remedy, we use the offline-online strategy from  \cite{MR4378546}, which we recall here to be self-contained.

Due to the Petrov-Galerkin version, for any $v \in V_H$, we can rewrite $\widetilde{a}_\mathfrak{k}(u_{{H,}\mathfrak{k}},v) $ as a sum of element components on a patch $U_\mathfrak{k}(T)$, \textit{i.e.},
\begin{equation*}
    \widetilde{a}_\mathfrak{k}(u_{{H,}\mathfrak{k}},v)  = \sum\limits_{T\in\mathcal{T}_H}\widetilde{a}_{\mathfrak{k},{T}}(u_{\scriptscriptstyle{H,}\mathfrak{k}},v) 
\end{equation*}
with
\begin{equation*}
    \widetilde{a}_{\mathfrak{k},{T}}(u_{{H,}\mathfrak{k}},v) := \int\limits_{U_\mathfrak{k}(T)}A(x)(\chi_T\nabla u_{{H},\mathfrak{k}} -\nabla(\mathscr{C}_{\mathfrak{k},T}u_{{H,}\mathfrak{k}}))(x)\cdot\nabla v(x)dx.
\end{equation*}
From this point onward, we assume that the mesh size $H$ is an integer multiple of the periodicity length $\varepsilon$. If $A=A_\varepsilon$, the local contribution $\widetilde a_{\mathfrak{k},T}(\cdot, \cdot)$  is exactly the same for every mesh element $T \in \mathcal{T}_H$. This motivates to pre-compute $\widetilde a_{\mathfrak{k},T}$ for certain defect configurations, but only one single fixed $T$, in the offline phase and then to recombine these quantities in a suitable manner for each realization of $A$.

\subsubsection{Offline phase}
\label{sec: offline}
Fix an element $T \in \mathcal{T}_H$ and let $J := \{k \in \mathbb{Z}^d: \varepsilon(k+Q) \subset U_\mathfrak{k}(T)\}$ be the index set of possible defects in the patch $U_\mathfrak{k}(T)$. We further denote its cardinality by $\mathcal{N}:= \text{card}J$. For the bijective mapping $\sigma:\{1, \dots,\mathcal{N}\} \to J$, define
\begin{equation*}
\label{off coef}
    A_i := \begin{cases} 
A_\varepsilon |_{U_\mathfrak{k}(T)}; & i=0 \\
A_\varepsilon |_{U_\mathfrak{k}(T)}+\chi_{\varepsilon(\sigma(i)+Q)}B_\varepsilon; & i= 1, \dots, \mathcal{N}.
\end{cases}
\end{equation*}
This is the offline ``basis'' of coefficients where $A_i$ is constructed from $A_\varepsilon$ by considering only a single defect.  Figure \ref{offline} shows the first four offline coefficients by using such a single defect on a single patch for the random checkerboard. This is similarly done for the random erasure by removing a single element from the perfectly periodic configuration.
\begin{figure}[h]
\centering
      \includegraphics[width=9.5cm, height=2.2cm]{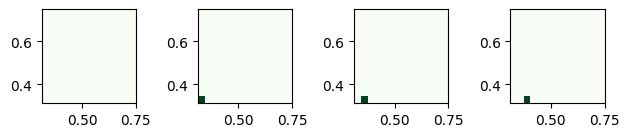}
      \caption[0.3\textwidth]{Offline coefficients $A_0, A_1, A_2, A_3$ for random checkerboard}
      \label{offline}
    \end{figure}

\noindent Next, for each offline coefficient, we pre-compute and store the LOD stiffness matrix contributions
\begin{equation*} 
\label{off stiff}
\widetilde{a}^i_{\mathfrak{k},{T}}(\varphi_j,\varphi_k)=\int\limits_{U_\mathfrak{k}(T)} A_i(x)(\chi_T\nabla\varphi_j-\nabla (\mathscr{C}_{\mathfrak{k},T}^i\varphi_j))(x)\cdot\nabla\varphi_k(x)dx
\end{equation*}
where $\mathscr{C}_{\mathfrak{k},T}^i:=\mathscr{C}_{\mathfrak{k},T}(A_i)$. We stress that the stiffnes matrix contribution for a fixed element $T$ is a coarse-scale object and inexpensive to store. At the same time, we can assemble the standard FE mass matrix.

\subsubsection{Online Phase}
\label{sec: online}
In the online phase, the global LOD stiffness matrix that corresponds to a true coefficient $A$ is approximated by a combination of offline quantities. Given a sample coefficient $A$, we locally write it as
\begin{equation*}
    A|_{U_\mathfrak{k}(T)}= \sum\limits_{i=0}^{\mathcal{N}}\mu_i A_i
\end{equation*}
for any $T \in \mathcal{T}_H$. The $\mu_i$'s are such that $\sum\limits_{i=0}^{\mathcal{N}}\mu_i=1$ for $\mu_i \in \mathbb{R}$, $i=0, \dots, \mathcal{N}$ and each $\mu_i$ is determined by the value of $\hat{b}^j_p$ for some $j$. Furthermore, $\mu_i\in \{0, 1\}$ for $i=1,\dots,\mathcal{N}$ with $\mu_0=1-\mathcal{N}_{def}$, where $\mathcal{N}_{def}$ denotes the number of defects on the patch $U_\mathfrak{k}(T)$. Then, we compute the offline-online LOD stiffness matrix contribution corresponding to an element $T$ by
\begin{equation}
\label{eqnn: 5.101}
    \widehat{a}_{\mathfrak{k},T}(\varphi_j, \varphi_k) = \sum\limits_{i=0}^{\mathcal{N}}\mu_i\widetilde{a}^i_{{\mathfrak{k}},{T}}(\varphi_j,\varphi_k).
\end{equation}
The global combined bilinear form $\widehat{a}$ is defined by
\begin{equation*}
\label{eqn: 6.34}
    \widehat{a}_\mathfrak{k}=\sum\limits_{T\in \mathcal{T}_H}\widehat{a}_{\mathfrak{k},T}.
\end{equation*}
The sum in \eqref{eqnn: 5.101} only contains $\mathcal{N}_{def}+1$ non-zero terms. Therefore, only a fraction of roughly $p$ terms need to be considered each time. Altogether, the problem is as follows:\\

\noindent \textbf{The Offline-Online MLOD problem (OLOD problem):} Find $(\hat{\lambda}_{H,\mathfrak{k}},\hat{u}_{H,\mathfrak{k}}) \in  \mathbb{R} \times V_H$ such that,
\begin{equation}
 \label{Eq:5.103}
    \widehat{a}_\mathfrak{k}(\hat{u}_{H,\mathfrak{k}}, v) = \hat{\lambda}_{H,\mathfrak{k}}\widehat{m}_\mathfrak{k}(\hat{u}_{H,\mathfrak{k}},v), \;\; \forall v \in V_H
\end{equation}
with the mass matrix $\widehat{m}_\mathfrak{k}(\cdot, \cdot)=\widetilde{m}_\mathfrak{k}(\cdot, \cdot) = m(\cdot, \cdot)$. 
The solution $\hat{u}_{H,\mathfrak{k}}$ is a good $L^2$ approximation to the exact solution $u$. In order to get an approximation in $\mathcal{H}^1$, we define the following multiscale solution, which can be efficiently computed online as well. We set
\begin{equation*}
 \hat{u}^\mathrm{ms}_{H,\mathfrak{k}} = \hat{u}_{H,\mathfrak{k}} -\widehat{\mathscr{C}}_\mathfrak{k}\hat{u}_{H,\mathfrak{k}} \;\;\;\;\; \text{with }\;\;\;\; \widehat{\mathscr{C}}_\mathfrak{k} := \sum\limits_{T\in \mathcal{T}_H}\sum\limits_{i=0}^\mathcal{N} \mu_i\mathscr{C}_{\mathfrak{k},T}^i.
\end{equation*}
This computation requires to store the correctors $\mathscr{C}_{\mathfrak{k},T}(A_i)$ in the offline phase, but again only for one fixed element.

\begin{remark}
Computationally, the offline phase is relatively expensive, but it only needs to be computed once independent of the number of samples $A$. The online phase compiles a real-time stiffness matrix from the information readily available in offline storage, making it very fast.
\end{remark}

\begin{remark}
Under the considered hypercube domain structure with periodic boundary conditions, the local LOD stiffness matrix for $A_\varepsilon$ and the selection of offline coefficients are exactly the same for every element making the computation far more efficient and reducing the required storage. For more complex domains or different boundary conditions, the method can be adapted by modifying the offline phase accordingly. This adaptation would involve computing and storing the matrix contributions for all possible patch configurations along with their corresponding offline coefficients. If the number of such configurations remains relatively small, as in the case of a highly structured mesh and domain, the additional computational cost may still be practical.
\end{remark}

\begin{remark}
    The possibility of using the same patch by transferring it through the whole domain contribute to a higher computational efficiency. While larger patches improve accuracy by capturing non-local interactions, they also increase storage and computational costs as the number of potential defect configurations in the offline phase grows exponentially with patch size. A recent development is the Super Localized Orthogonal Decomposition (SLOD) method proposed in \cite{MR4550317}, which enables the use of smaller patches without significantly compromising accuracy. While combining the offline-online strategy with the SLOD method could be attractive, our current approach is particularly based on the Petrov-Galerkin LOD formulation, and the matrix assembly is tailored to the standard LOD method. Therefore, it is not straightforward to apply our strategy to the SLOD framework in its current form.
\end{remark}

\noindent The consistency error $ \eta_\mathfrak{k}$ defined by
\begin{equation*}
    \eta_\mathfrak{k} := \sup\limits_{v_1\in V_H\backslash\{0\}}\sup\limits_{v_2\in V_H\backslash \{0\}} \frac{\abs{(\widetilde{a}_\mathfrak{k}-\widehat{a}_\mathfrak{k})(v_1,v_2)}}{\norm{v_1}_{\mathcal{A}}\norm{v_2}_{\mathcal{A}}},
\end{equation*}
and the error between correctors $\mathscr{C}_\mathfrak{k}$ and $\widehat{\mathscr{C}}_\mathfrak{k}$ are important results for the error analysis in Section \ref{sec:3}, and have been analysed in \cite{MR4378546}. There, the following upper bound on the consistency $\widetilde{a}_\mathfrak{k}-\widehat{a}_\mathfrak{k}$, that is computable in a posteriori manner, is established.

\begin{theorem}
\label{Thm:prelim}
Define for any $T\in \mathcal{T}_H$
\begin{equation*}
    E^2_T := \max\limits_{v\in
         V_H: v|_T} \frac{\norm{\left(A^{1/2}-A^{-1/2}\overline{A}\right)\chi_T\nabla v - \sum\limits_{i=0}^\mathcal{N}\mu_i\left(A^{1/2}-A^{-1/2}A_i\right)\nabla \left(\mathscr{C}_{\mathfrak{k},T}^iv\right)}^2_{L^2(U_\mathfrak{k}(T))}}{\norm{v}^2_{A,T}}
\end{equation*}
with $\overline{A}= \sum\limits_{i=0}^{\mathcal{N}}\mu_iA_i$. Then, for any $v_1, v_2 \in V_H$ it holds that
\begin{equation}
\label{EQN: 5.108}
    \abs{\left(\widehat{a}_\mathfrak{k}-\widetilde{a}_\mathfrak{k}\right)\left(v_1,v_2\right)} \lesssim \mathfrak{k}^{d/2}\left(\max\limits_{T\in \mathcal{T}_H}E_T\right)\norm{v_1}_\mathcal{A}\norm{v_2}_\mathcal{A}.
\end{equation}
Furthermore, for any $v \in V_H$, the error between correctors is bounded by
\begin{equation}
\begin{split}
\label{eq:7109}
    \norm{\left({\mathscr{C}_\mathfrak{k}}-\widehat{\mathscr{C}}_\mathfrak{k}\right){v}}^2_\mathcal{A}
    \lesssim \;\mathfrak{k}^{d}\left(\max_{T\in \mathcal{T}_H}E_T\right)^2\norm{v}^2_\mathcal{A}.
\end{split}
\end{equation}
\end{theorem}

\noindent The well-possedness of the OLOD follows due to the coercivity of $\widehat{a}_\mathfrak{k}(v,v)$ given by
\begin{equation}
    \begin{split}
        \widehat{a}_\mathfrak{k}(v,v) \geq \left(c_1 - c_2g(\mathfrak{k})- \eta_\mathfrak{k}\right) \norm{v}^2_{\mathcal{A}}
        \label{equ:7}
    \end{split}
\end{equation}
where the constants $c_1$ and $c_2$ only depend on the spectral bounds and the shape regularity. Clearly, there exist $\mathfrak{k}_0 < \mathfrak{k}$ and $\eta_0^{\mathfrak{k}} > \eta_\mathfrak{k}$ such that $c_1 - c_2g(\mathfrak{k})-\eta_\mathfrak{k}$ is bounded below by a positive constant.

\subsection{An alternate offline-online strategy}
\label{subsec:alternate}

In this section, we define an enhanced offline-online strategy for random checkerboard-type coefficients, particularly when $Q=[0,1]^d$. Note that, if $Q$ is a proper subset of $[0,1]^d$, the sum constraint $\sum_i \mu_i=1$ is necessary to guarantee the correct value of the coefficient $A$ at points $x$ where defects can never occur. One example is the case of random erasure and points $x$ corresponding to the value $A_\mathrm{per}$ independent of the realization. That means, in random erasure, the sum constraint $\sum_i \mu_i= 1$ is crucial to obtain the correct multiscale basis functions that preserve the mean value of the solution over each coarse element and one cannot relax this restriction for the underlying random erasure model. However, in the random checkerboard case, the sum constraint $\sum_i \mu_i=1$ is merely used to guarantee uniqueness of the $\mu_i$. The alternate method we propose takes advantage of this by relaxing the sum constraint to be a real number $s$ that is adapted to the underlying distribution to obtain more accurate results. 
In the explained random checkerboard example, we use the Bernoulli distribution. However, other probability distributions on the interval $[\alpha, \beta]$, such as the uniform distribution, are possible in general as well. To this end, we consider the coefficient $A$  of the following form
    \begin{equation*}
         A(x, \omega) = \sum\limits_{j \in J}\chi_{\varepsilon(j+Q)}(x)\tilde{\beta}_j(\omega)
    \end{equation*}
   for the alternate approach where $\tilde{\beta}_j$ are independent and identically distributed random variables with $\tilde{\beta}_j\in \{\alpha, \beta\}$. The index set $J$ and the offline coefficients $A_i$ are defined as in Section \ref{sec: offline}. We emphasize that, for the random checkerboard case, this is just a re-writing of the coefficient.

\noindent In contrast to the sum constraint one of coefficients in the online phase defined in Section \ref{sec: online}, we  now assume $\sum\limits_{i=0}^{\mathcal{N}}\mu_i=s$ for some $s\in \mathbb{R}^+$ with the property
\begin{equation*}
    A|_{U_\mathfrak{k}(T)}= \sum\limits_{i=0}^{\mathcal{N}}\mu_i A_i \text{ for any }T \in \mathcal{T}_H.
\end{equation*}
The idea of the new strategy is to find the value $s$ using a heuristic approach to minimize the consistency error in the one-dimensional case. More precisely, as discussed in \cite{MR4378546}, in the one-dimensional setting with the nodal interpolation operator and patches consisting only of a single element $T$, the Petrov-Galerkin LOD corresponds to a finite element method with the piecewise harmonic mean as coefficient. That is
\begin{equation*}
    \widetilde{a}^T_\mathfrak{k}({u}_{{H}},v_H) = \int\limits_T A_{\text{harm}}|_T\nabla u_H \nabla v_H dx
\end{equation*}
\noindent where the element-wise constant coefficient $ A_{\text{harm}}$ is given by
\begin{equation*}
    A_{\text{harm}}|_T := \left(\frac{1}{\abs{T}}\int_TA^{-1}dx\right)^{-1}=\frac{\abs{T}}{\varepsilon}\frac{1}{\sum\limits_{i=1}^{\mathcal{N}}\big(\tilde{\beta}_i\big)^{-1}}.
\end{equation*}
We can do the same for our offline-online LOD stiffness matrix, where we denote the coefficient by $A^\mu$. In the following, we calculate $A_\mathrm{harm}$ and $A^{\mu}_\mathrm{harm}$ explicitly and then analyze their error. We obtain $A^\mu$ as
\begin{equation}
\label{eq: Amu}
    A^\mu_{\text{harm}}|_T =\sum\limits_{i=0}^{\mathcal{N}}\mu_iA^i_{\text{harm}}|_T = \mu_0A_0 + \sum\limits_{i=1}^{\mathcal{N}}\mu_iA^i_{\text{harm}}|_T,
    \end{equation}
where the harmonic mean of $A_i$, $A^i_{\text{harm}}$ for all $ i \geq 1$ in the case of random checkerboard can be written as
\begin{equation}
\label{eq: Aharm}
    A^i_{\text{harm}} = \frac{\mathcal{N}}{\left(\mathcal{N}-1\right)\left(1/\alpha\right)+(1)\left(1/\beta\right)}.
\end{equation}
The number of $\varepsilon$ blocks in an element $T$ denoted by $\mathcal{N}$ is given by $\abs{T}/\varepsilon$. We now need to find the correct representations of $\mu_0$ and $\mu_i$, $i\geq 1$. This can easily be computed from
\begin{equation*}
      \sum\limits_{i \in J}\chi_{\varepsilon(i+Q)}(x)\tilde{\beta}_i(\omega) = \sum\limits_{i=0}^{\mathcal{N}}\mu_i A_i = A_0\mu_0 + \sum\limits_{i=1}^{\mathcal{N}}\mu_i A_i,
\end{equation*}
for $A_\mathrm{per}= \alpha$ and $B_\mathrm{per}=\beta - \alpha$ in the random checkerboard set-up, in which the underlying linear system of equations assigns $\tilde{\beta}_i$ value for each $\varepsilon$ block on the patch via the characteristic function $\chi_{\varepsilon\left(i+Q\right)}(x)$. Further observe that $A_0$ is nothing but the background coefficient $A_{\text{per}}$. Hence, we deduce
\begin{equation}
\begin{split}
\label{eq: mu}
    \mu_i = \frac{\tilde{\beta}_i-\alpha s}{\beta - \alpha}, \;\; i \geq 1, \hspace{2cm}
    \mu_0 = \frac{s\left(\beta - \alpha\right)- \sum\limits_{i=1}^{\mathcal{N}}\big(\tilde{\beta}_i-\alpha s\big)}{\beta-\alpha}.
    \end{split}
\end{equation}

\noindent By substituting \eqref{eq: Aharm} and  \eqref{eq: mu} into \eqref{eq: Amu}, we arrive at
\begin{equation*}
\begin{split}
    A^\mu_{\text{harm}}|_T 
    &=\frac{\abs{T}}{\varepsilon}\frac{1}{\beta-\alpha}\Biggl(\frac{s(\beta-\alpha)-\sum\limits_{i=1}^{\mathcal{N}}\big(\tilde{\beta}_i-\alpha s\big)}{\mathcal{N}\left(1/\alpha\right)}+\frac{\sum\limits_{i=1}^{\mathcal{N}}\big(\tilde{\beta}_i-\alpha s\big)}{\left(\mathcal{N}-1\right)\left(1/\alpha\right)+\left(1/\beta\right)}\Biggr)\\
    &=\frac{\abs{T}}{\beta-\alpha}\Bigg(\frac{s\left(\beta-\alpha\right)-\sum\limits_{i=1}^{\mathcal{N}}\big(\tilde{\beta}_i-\alpha s\big)}{\mathcal{N}\varepsilon/\alpha}+\frac{\sum\limits_{i=1}^{\mathcal{N}}\big(\tilde{\beta}_i-\alpha s\big)}{\left(\mathcal{N}-1\right)\varepsilon/\alpha+{\varepsilon}/{\beta}}\Bigg)
    \end{split}
\end{equation*}
With this, we obtain for the error that

\begin{equation*}
    \begin{split}
        A_{\text{harm}}|_T-A^\mu_{\text{harm}}|_T &= \frac{\abs{T}}{\varepsilon}\left[\frac{1}{\sum\limits_{i=1}^{\mathcal{N}}\big(\tilde{\beta}_i\big)^{-1}}-\frac{1}{\beta-\alpha}\Bigg(\frac{s\left(\beta-\alpha\right)-\sum\limits_{i=1}^{\mathcal{N}}\big(\tilde{\beta}_i-\alpha s\big)}{\mathcal{N}}\alpha + \frac{\sum\limits_{i=1}^{\mathcal{N}}\big(\tilde{\beta}_i-\alpha s\big)}{\frac{\left(\mathcal{N}-1\right)\beta+\alpha}{\alpha\beta}}\Bigg)\right]\\
        &=H\big(\tilde{\beta}\big)-s\alpha+\frac{\alpha}{\beta-\alpha}\Bigg(\frac{\sum\limits_{i=1}^\mathcal{N}\big(\tilde{\beta}_i-\alpha s\big)}{\mathcal{N}}\Bigg)\mathcal{N}-\frac{\mathcal{N}^2\alpha\beta}{\left(\beta-\alpha\right)}\Biggl[\Biggl(\frac{\sum\limits_{i=1}^\mathcal{N}\big(\tilde{\beta}_i-\alpha s\big)}{\mathcal{N}}\Biggr)\frac{1}{\left(\mathcal{N}-1\right)\beta+\alpha}\Biggr]\\
        &= H\big(\tilde{\beta}\big)-s\alpha+\frac{M\big(\tilde{\beta}-\alpha s\big)}{\left(\beta-\alpha \right)} \frac{\alpha \mathcal{N}\left(\alpha -\beta \right)}{\left(\mathcal{N}-1\right)\beta+\alpha}\\
        &=H\big(\tilde{\beta}\big)-s\alpha - \alpha M\big(\tilde{\beta}-\alpha s\big)\frac{1}{\left(1-1/\mathcal{N}\right)\beta+\alpha/\mathcal{N}}
    \end{split}
\end{equation*}
with $H\big(\tilde{\beta}\big)=\mathcal{N}\left(\sum\limits_{i=1}^{\mathcal{N}}\big(\tilde{\beta}_i\big)^{-1}\right)^{-1}$ and $ M\big(\tilde{\beta}-\alpha s\big)=\left({\sum\limits_{i=1}^\mathcal{N}\big(\tilde{\beta}_i-\alpha s\big)}\right)/{\mathcal{N}}$. Our goal is to make this error as small as possible because this will imply a small consistency error. We consider the limit $\mathcal N\to \infty$ which corresponds to vanishing periodicity length $\varepsilon \to 0$, and try to find $s$ such that the error vanishes in this limit. In other words, we find $s$ in such a way that in the homogenization limit, there is no consistency error and our offline-online strategy retrieves the homogenized coefficient. Denote
\begin{equation*}
\label{mean}
  \overline{H} = \lim\limits_{\mathcal{N}\to \infty} H\big({\tilde{\beta}}\big) \;\;\;\text{ and }\;\;\; \overline{M} = \lim\limits_{\mathcal{N}\to \infty} M\big({\tilde{\beta}}\big).
\end{equation*} 
Note that $\overline{H}$ and $\overline{M}$ can be deduced from the probability law of $\tilde \beta$ as certain expectations by the law of large numbers.
Therefore, as $\mathcal{N}\to \infty$, we obtain
\begin{equation*}
  A_{\text{harm}}|_T-A^\mu_{\text{harm}}|_T = \overline{H}-s\alpha -\frac{\alpha}{\beta}\left(\overline{M}-\alpha s\right) =0.
\end{equation*}

\noindent Hence the optimal $s$ is obtained when
\begin{equation*}
    s = \frac{\beta \overline{H} - \alpha \overline{M}}{\alpha (\beta - \alpha)}.
\end{equation*}

\begin{example}
\label{ex:2.7}
If $\tilde{\beta}$ is Bernoulli distributed with probability $p$ of having a $\beta$ and  probability $1-p$ for an $\alpha$, then we have
\begin{equation*}
    \overline{M} = \alpha\left(1-p\right)+ \beta p \;\;\;\text{ and } \;\;\;  \overline{H} = \frac{\alpha \beta}{\beta+\left(\alpha-\beta\right)p}.
\end{equation*}
Then, $s$ is given by
\begin{equation*}
    s  = 1+\frac{\left(\beta-\alpha\right)p^2}{\beta + \left(\alpha-\beta\right)p}.
\end{equation*}
 Since the alternate $s$-value deviates more from $1$ as $p$ increases, we expect the influence of our new strategy to be greater for larger $p$. This will be studied in more detail numerically in Section \ref{sec:5}, specifically numerical experiments for values of $s$ in accordance with the Example \ref{ex:2.7} are illustrated in Figures \ref{fig:com1} and \ref{fig:com2} in Section \ref{sec:5} below.
\end{example}

\section{Preliminaries for the error analysis}
\label{sec:3}

Define the eigenspace spanned by orthonormal eigenfunctions $u^{(1)}, \cdots, u^{(l)}$ by the $l$-lowermost eigenvalues $\lambda^{(1)}, \cdots, \lambda^{(l)}$ of \eqref{eqnn: 2.7} by
\begin{equation*}
    {E^{(l)}}(\lambda^{(l)}) := \text{span}\left\{u^{(1)}, \cdots, u^{(l)}\right\}. 
\end{equation*}

\noindent In the following analysis, we rely on the eigenvalue analysis by using compact operators and spectral theory introduced by Babu\v{s}ka and Osborn in \cite{MR1893416}. In the next section, we derive several preliminary results from compact operators useful in the error analysis.

Suppose $T$, $T^1_H: \mathcal{H}^1_\#(\Omega) \to \mathcal{H}^1_\#(\Omega)$ be the solution operators corresponding to the continuous eigenvalue problem written in \eqref{eqnn: 2.7} and a certain LOD based multiscale problem (say $M1$) respectively, and are defined as follows.

\noindent Find $Tf\in \mathcal{H}^1_\#(\Omega) $ such that
\begin{equation*}
\label{E: 5.67}
  a(Tf, v) = m(f,v), \;\; \forall v\in \mathcal{H}^1_\#(\Omega).
\end{equation*}

\noindent Find $T^1_Hf\in V_H \subset \mathcal{H}^1_\#(\Omega)$ such that
\begin{equation}
\label{E: 5.68}
    a_1\left(T^1_Hf, v\right) = m_1\left(f,v\right), \;\; \forall v\in V_H.
\end{equation}
By Babu\v{s}ka and Osborn theory \cite{MR1115240}, the eigenvalues of $T$ are the reciprocals of the eigenvalues of the equation \eqref{eqnn: 2.7} and both problems produce the same eigenfunctions. Similarly, $T^1_H$ has the reciprocal eigenvalues (when non-zero) and equal eigenfunctions of $M1$. Our error analysis of the offline-online strategy will rely on perturbation arguments, namely interpreting the offline-online LOD as variant of the LOD with different stiffness and mass matrix. Therefore, we also define an approximate solution operator $T^2_H$ possessing the reciprocal eigenvalues and similar eigenfunctions of a second LOD based multiscale problem (say $M2$) by the following. \newline

\noindent Find $T^2_Hf \in V_H$ such that
\begin{equation}
\label{E: 5.69}
    a_2\left(T^2_Hf,v\right) = m_2\left(f,v\right), \;\; \forall v\in V_H.
\end{equation}

\noindent In all cases above $f \in \mathcal{H}^1_\#(\Omega)$. For $f\in V_H$, the function $T^2_Hf$ is defined and is taken as a map on $V_H$. While $T$ is viewed as an operator on $\mathcal{H}^1_\#(\Omega)$, the operator $T^1_H$ can be considered either to be on $\mathcal{H}^1_\#(\Omega)$ or on $V_H$. For the perturbation analysis, we need to relate the error $T^1_H-T^2_H$ to the errors in $m_1 - m_2$ and $a_1 - a_2$ in the following.

\begin{lemma}
\label{L: 6.2.1}
Suppose $T^1_Hf$ and $T^2_Hf$ are approximated solutions to the problem \eqref{eqnn: 2.7} solved by the numerical methods as defined in problems \eqref{E: 5.68} and \eqref{E: 5.69} respectively. Then it holds that
\begin{equation*}
\begin{split}
     \hspace{2cm} \norm{T^1_Hf-T^2_Hf}^2_{\mathcal{A}}  & \lesssim \abs{\left(m_1-m_2\right)\left(f,T^1_Hf-T^2_Hf\right)}\\
        &\hspace{1.5cm} + \abs{\left(a_2-a_1\right)\left(T^2_Hf, T^1_Hf-T^2_Hf\right)}.
\end{split}
\end{equation*}
\end{lemma}

\begin{proof}
\noindent We deduce together with \eqref{E: 5.68} and \eqref{E: 5.69} that
\begin{align*}
\label{eqn: comp}
       \norm{T^1_Hf-T^2_Hf}^2_{\mathcal{A}}   &\lesssim a_1\left(T^1_Hf-T^2_Hf,T^1_Hf-T^2_Hf\right)\NNN
       &= a_1\left(T^1_Hf,T^1_Hf-T^2_Hf\right)-a_2\left(T^2_Hf,T^1_Hf-T^2_Hf\right) \NNN
         & \hspace{2cm} +a_2\left(T^2_Hf,T^1_Hf-T^2_Hf\right)- a_1\left(T^2_Hf,T^1_Hf-T^2_Hf\right)\NNN
          & \leq \abs{(m_1-m_2)(f,T^1_Hf-T^2_Hf)} + \abs{(a_2-a_1)(T^2_Hf, T^1_Hf-T^2_Hf)}.\qedhere
    \end{align*}
\end{proof}

\section{Error analysis}
\label{sec:4}
In this section, we derive error estimates for eigenfunctions and eigenvalues for the MLOD and the OLOD methods. Recall that in the MLOD formulation, we have considered the trial space to be the multiscale space and the test space to be the classical finite element space on the coarse mesh. Consequently, the stiffness and mass matrix in the MLOD are based on the coarse FE space and different from the ones used in the Galerkin LOD method in \cite{MR3343928}. We therefore start  with several preliminary results particularly estimating the errors from these changes. Note that our final goal is to obtain estimates for the OLOD method, for which --  similar to \cite{MR1035176} --- the MLOD method is introduced and analyzed as an intermediate step.

\subsection{Preliminary estimates for the MLOD}
\label{sec: 4.1}
The following two lemmas derive error estimate for choosing the FEM mass matrix and Petrov-Galerkin LOD stiffness matrix with respect to that of the LOD method. 
\begin{lemma}
For any $v_1,v_2 \in V=\mathcal{H}^1_\#(\Omega)$,  it holds that
\label{Lem: 7.75}
\begin{equation*}
    \begin{split}
         \frac{\abs{\left({m}_\mathfrak{k}-\widetilde{m}_\mathfrak{k}\right)(v_1,v_2)}}{\norm{v_1}_{\mathcal{A}}\norm{v_2}_{\mathcal{A}}} &\lesssim H^2.
    \end{split}
\end{equation*}
\end{lemma}

\begin{proof}
We can deduce that
\begin{equation*}
    \begin{split}
        \abs{({m}_\mathfrak{k}-\widetilde{m}_\mathfrak{k})(v_1,v_2)}
        &=\abs{m(v_1-\mathscr{C}_\mathfrak{k}v_1, v_2-\mathscr{C}_\mathfrak{k}v_2)-m(v_1,v_2)}\\
        &=\abs{m(v_1,v_2)-m(v_1,\mathscr{C}_\mathfrak{k}v_2)-m(\mathscr{C}_\mathfrak{k}v_1,v_2)+m(\mathscr{C}_\mathfrak{k}v_1,\mathscr{C}_\mathfrak{k}v_2)-m(v_1,v_2)}\\
        &\leq \abs{m(v_1,\mathscr{C}_\mathfrak{k}v_2)+ m(v_2,\mathscr{C}_\mathfrak{k}v_1)}+\abs{m(\mathscr{C}_\mathfrak{k}v_1,\mathscr{C}_\mathfrak{k}v_2)}.
    \end{split}
\end{equation*}
Since $\mathscr{C}_\mathfrak{k}(\cdot) \in W \subset V$, by Lemma \ref{Lem: 5.1.4}, we get that
\begin{equation*}
\begin{split}
\label{eqnnn: 5.68}
    \abs{\left({m}_\mathfrak{k}-\widetilde{m}_\mathfrak{k}\right)(v_1,v_2)} & \lesssim H^2\left(\norm{v_1}_{\mathcal{A}}\norm{\mathscr{C}_\mathfrak{k}v_2}_{\mathcal{A}}+\norm{v_2}_{\mathcal{A}}\norm{\mathscr{C}_\mathfrak{k}v_1}_{\mathcal{A}} +\norm{\mathscr{C}_\mathfrak{k}v_1}_{\mathcal{A}}\norm{\mathscr{C}_\mathfrak{k}v_2}_{\mathcal{A}}\right) \\
    & \leq H^2\Big(\norm{v_1}_{\mathcal{A}}\left(\norm{(\mathscr{C}-\mathscr{C}_\mathfrak{k})v_2}_{\mathcal{A}}+\norm{\mathscr{C}v_2}_{\mathcal{A}}\right)+\norm{v_2}_{\mathcal{A}}\left(\norm{(\mathscr{C}-\mathscr{C}_\mathfrak{k})v_1}_{\mathcal{A}}+\norm{\mathscr{C}v_1}_{\mathcal{A}}\right)\\
   & \hspace{4.8cm}+ \left( \norm{\left(\mathscr{C}-\mathscr{C}_\mathfrak{k}\right)v_1}_{\mathcal{A}}+\norm{\mathscr{C}v_1}_{\mathcal{A}}\right)\left(\norm{(\mathscr{C}-\mathscr{C}_\mathfrak{k})v_2}_{\mathcal{A}}+\norm{\mathscr{C}v_2}_{\mathcal{A}}\right)\Big) \\
 & \lesssim H^2\left(3 + 4g(\mathfrak{k}) + g^2(\mathfrak{k})\right)\norm{v_1}_{\mathcal{A}}\norm{v_2}_{\mathcal{A}}.
    \end{split}
\end{equation*}

The last step follows by Theorem \ref{thm:5.2.4} and the stability of $\mathscr{C}$ in the energy norm. The notion $g^2(\mathfrak{k})$ is used for $(g(\mathfrak{k}))^2$. Since $H^2$ dominates the order, we deduce the assertion.
\end{proof}

\begin{lemma}
For any $v_1, v_2 \in V=\mathcal{H}^1_\#(\Omega)$, we have
\label{Lem: 764}
\begin{equation*}
    \frac{\abs{({a}_\mathfrak{k}-\widetilde{a}_\mathfrak{k})(v_1,v_2)}}{\norm{v_1}_{\mathcal{A}}\norm{v_2}_{\mathcal{A}}}\lesssim  g(\mathfrak{k}).
\end{equation*}
\end{lemma}

\begin{proof}
We use the definitions of $a_\mathfrak{k}\left(\cdot, \cdot\right)$ and $\widetilde{a}_\mathfrak{k}(\cdot, \cdot)$ to obtain
\begin{equation*}
    \begin{split}
\left({a}_\mathfrak{k}-\widetilde{a}_\mathfrak{k}\right)\left(v_1,v_2\right) &=  a\left((1-\mathscr{C}_\mathfrak{k})v_1,(1-\mathscr{C}_\mathfrak{k})v_2\right) - a\left((1-\mathscr{C}_\mathfrak{k})v_1,v_2\right)\\
      & = a\left((1-\mathscr{C}_\mathfrak{k})v_1,-\mathscr{C}_\mathfrak{k}v_2\right)\\
      & = a\left((1-\mathscr{C})v_1+(\mathscr{C}-\mathscr{C}_\mathfrak{k})v_1,-\mathscr{C}_\mathfrak{k}v_2\right) \\
      & =  a((\mathscr{C}-\mathscr{C}_\mathfrak{k})v_1,-\mathscr{C}_\mathfrak{k}v_2) \\
      & = a((\mathscr{C}-\mathscr{C}_\mathfrak{k})v_1,(\mathscr{C}-\mathscr{C}_\mathfrak{k})v_2)  + a((\mathscr{C}-\mathscr{C}_\mathfrak{k})v_1,-\mathscr{C}v_2) 
    \end{split}
\end{equation*}
The second last equality follows by the $a$-orthogonality property of the correction operator $\mathscr{C}$. Then by Theorem \ref{thm:5.2.4} and the stability of $\mathscr{C}$ in the energy norm yields that
\begin{equation*}
    \begin{split}
({a}_\mathfrak{k}-\widetilde{a}_\mathfrak{k})(v_1,v_2) 
        & \lesssim \norm{(\mathscr{C}-\mathscr{C}_\mathfrak{k})v_1}_{\mathcal{A}}\norm{(\mathscr{C}-\mathscr{C}_\mathfrak{k})v_2}_{\mathcal{A}} + \norm{(\mathscr{C}-\mathscr{C}_\mathfrak{k})v_1}_{\mathcal{A}}\norm{v_2}_{\mathcal{A}}\\
        & \lesssim \norm{(\mathscr{C}-\mathscr{C}_\mathfrak{k})v_1}_{\mathcal{A}}\left(\norm{(\mathscr{C}-\mathscr{C}_\mathfrak{k})v_2}_{\mathcal{A}} + \norm{v_2}_{\mathcal{A}} \right)\\
         & \lesssim g(\mathfrak{k}) \left(g(\mathfrak{k})+1 \right) \norm{v_1}_{\mathcal{A}}\norm{v_2}_{\mathcal{A}}.
    \end{split}
\end{equation*}
Clearly, $g(\mathfrak{k})$ decides the order of convergence.
\end{proof}

\begin{remark}
\noindent If taken $\mathfrak{k}$ large enough such that $g(\mathfrak{k}) \approx H^2$, Lemmas \ref{Lem: 7.75} and \ref{Lem: 764} tells us that the error between mass matrices due to the selection of the FEM mass matrix is the same as the error due to the Petrov-Galerkin localization in the stiffness matrix. 
\end{remark}

\noindent In the next lemma, we introduce a strategy for the eigenvalue error estimates for the MLOD method. 

\begin{lemma}
\label{Lem: 7.4.4}
Suppose $(\lambda, u)$ is an eigenpair of \eqref{eqnn: 2.7}. Further let, $\norm{v^\mathrm{ms}_{H,\mathfrak{k}}}_{L^2(\Omega)}=1$ with $v^\mathrm{ms}_{H,\mathfrak{k}}, = v_{H,\mathfrak{k}} - \mathscr{C}_\mathfrak{k}v_{H,\mathfrak{k}}$  and $\widetilde{m}_\mathfrak{k}(v_{H,\mathfrak{k}},v_{H,\mathfrak{k}}) \neq 0$ for $v_{H,\mathfrak{k}} \in V_H \subset \mathcal{H}^1_\#(\Omega)$ and $v^\mathrm{ms}_{H,\mathfrak{k}} \in V^\mathrm{ms}_{H,\mathfrak{k}} \subset \mathcal{H}^1_\#(\Omega)$ . Consider the MLOD formulation of \eqref{eqnn: 2.7} given in \eqref{eqnn: 5.399}. Then,
\begin{equation*}
\begin{split}
    \frac{\widetilde{a}_\mathfrak{k}(v_{H,\mathfrak{k}},v_{H,\mathfrak{k}})}{\widetilde{m}_\mathfrak{k}(v_{H,\mathfrak{k}},v_{H,\mathfrak{k}})}-\lambda \leq & \norm{u-v^\mathrm{ms}_{H,\mathfrak{k}}}^2_{\mathcal{A}}-\lambda\norm{u-v^\mathrm{ms}_{H,\mathfrak{k}}}^2_{L^2(\Omega)} + \norm{(\mathscr{C}-\mathscr{C}_\mathfrak{k})v_{H,\mathfrak{k}}}_{\mathcal{A}}\norm{v_{H,\mathfrak{k}}}_{\mathcal{A}}\\
         & \hspace{2cm}+ \frac{\widetilde{a}_\mathfrak{k}(v_{H,\mathfrak{k}},v_{H,\mathfrak{k}})}{\widetilde{m}_\mathfrak{k}(v_{H,\mathfrak{k}},v_{H,\mathfrak{k}})}\left({m}_\mathfrak{k}(v_{H,\mathfrak{k}},v_{H,\mathfrak{k}})-\widetilde{m}_\mathfrak{k}(v_{H,\mathfrak{k}},v_{H,\mathfrak{k}})\right).
\end{split}
\end{equation*}
\end{lemma}
\begin{proof}
Observe that
\begin{equation}
\label{eqnn: 7.76}
    \norm{u-v^\mathrm{ms}_{H,\mathfrak{k}}}^2_{\mathcal{A}}-\lambda\norm{u-v^\mathrm{ms}_{H,\mathfrak{k}}}^2_{L^2(\Omega)}=\norm{u}^2_{\mathcal{A}}-2a(u,v^\mathrm{ms}_{H,\mathfrak{k}})+\norm{v^\mathrm{ms}_{H,\mathfrak{k}}}^2_{\mathcal{A}}-\lambda\norm{u}_{L^2(\Omega)}^{2}+2\lambda m(u,v^\mathrm{ms}_{H,\mathfrak{k}})-\lambda \norm{v^\mathrm{ms}_{H,\mathfrak{k}}}^2_{L^2(\Omega)}.
\end{equation}
Since $a(u,v^\mathrm{ms}_{H,\mathfrak{k}})=\lambda m(u,v^\mathrm{ms}_{H,\mathfrak{k}})$ holds by \eqref{eqnn: 2.7} and we have that $\norm{u}^2_{\mathcal{A}}=a(u,u)=\lambda m(u,u)= \lambda \norm{u}^2_{L^2(\Omega)}$, we get by the normalization $\norm{v^\mathrm{ms}_{H,\mathfrak{k}}}_{L^2(\Omega)}=1$, and \eqref{eqnn: 7.76} that
\begin{equation}
\label{eqn: 777}
   \norm{u-v^\mathrm{ms}_{H,\mathfrak{k}}}^2_{\mathcal{A}}-\lambda\norm{u-v^\mathrm{ms}_{H,\mathfrak{k}}}^2_{L^2(\Omega)} = a(v^\mathrm{ms}_{H,\mathfrak{k}},v^\mathrm{ms}_{H,\mathfrak{k}})-\lambda.
\end{equation}

\noindent Then we deduce that
\begin{equation*}
    \begin{split}
a\left(v^\mathrm{ms}_{H,\mathfrak{k}},v^\mathrm{ms}_{H,\mathfrak{k}}\right) & = a\left(v^\mathrm{ms}_{H,\mathfrak{k}},(1-\mathscr{C}_\mathfrak{k})v_{H,\mathfrak{k}}\right)\\
       & = a\left(v^\mathrm{ms}_{H,\mathfrak{k}},v_{H,\mathfrak{k}}\right) -  a\left((1-\mathscr{C}_\mathfrak{k})v_{H,\mathfrak{k}} ,\mathscr{C}_\mathfrak{k}v_{H,\mathfrak{k}} \right)\\
        & = a\left(v^\mathrm{ms}_{H,\mathfrak{k}},v_{H,\mathfrak{k}} \right) -  a\left(v_{H,\mathfrak{k}}-\mathscr{C}v_{H,\mathfrak{k}} ,\mathscr{C}_\mathfrak{k}v_{H,\mathfrak{k}}\right) - a\left((\mathscr{C}-\mathscr{C}_\mathfrak{k})v_{H,\mathfrak{k}} ,\mathscr{C}_\mathfrak{k}v_{H,\mathfrak{k}}\right).
    \end{split}
\end{equation*}

\noindent Due to the $a$-orthogonality of $\mathscr{C}$, the middle term vanishes. Then the stability of $\mathscr{C}_\mathfrak{k}$ in the energy norm for $v_{H,\mathfrak{k}} \in V_H$ yields
\begin{equation}
    \begin{split}
    \label{EQNN:784}
a\left(v^\mathrm{ms}_{H,\mathfrak{k}},v_{H,\mathfrak{k}}\right) 
      & \leq  a\left(v^\mathrm{ms}_{H,\mathfrak{k}},v^\mathrm{ms}_{H,\mathfrak{k}}\right) +  \norm{\left(\mathscr{C}-\mathscr{C}_\mathfrak{k}\right)v_{H,\mathfrak{k}}}_{\mathcal{A}}\norm{v_{H,\mathfrak{k}}}_{\mathcal{A}}.
\end{split}
\end{equation}

\noindent Together with \eqref{eqn: 777}, we arrive at
\begin{equation*}
    \begin{split}
    \label{eq:7.86}
a\left(v^\mathrm{ms}_{H,\mathfrak{k}},v_{H,\mathfrak{k}} \right)-\lambda \leq \norm{u-v^\mathrm{ms}_{H,\mathfrak{k}}}^2_{\mathcal{A}}-\lambda\norm{u-v^\mathrm{ms}_{H,\mathfrak{k}}}^2_{L^2(\Omega)} + \norm{\left(\mathscr{C}-\mathscr{C}_\mathfrak{k}\right)v_{H,\mathfrak{k}}}_{\mathcal{A}}\norm{v_{H,\mathfrak{k}}}_{\mathcal{A}}.
    \end{split}
\end{equation*}

\noindent Recall that $\widetilde{a}_\mathfrak{k}(\cdot,\cdot) = a\left((\cdot)^\mathrm{ms}_\mathfrak{k}, (\cdot)\right)$ and $\widetilde{m}_\mathfrak{k}(\cdot, \cdot)=m((\cdot), (\cdot))=\norm{(\cdot)}^2_{L^2(\Omega)}$. Then, together with the above inequality, we obtain that
\begin{equation*}
    \begin{split}
          \frac{\widetilde{a}_\mathfrak{k}(v_{H,\mathfrak{k}},v_{H,\mathfrak{k}})}{\widetilde{m}_\mathfrak{k}(v_{H,\mathfrak{k}},v_{H,\mathfrak{k}})} - \lambda &= [ \widetilde{a}_\mathfrak{k}(v_{H,\mathfrak{k}},v_{H,\mathfrak{k}}) -\lambda]+\left[    \frac{\widetilde{a}_\mathfrak{k}(v_{H,\mathfrak{k}},v_{H,\mathfrak{k}})}{\widetilde{m}_\mathfrak{k}(v_{H,\mathfrak{k}},v_{H,\mathfrak{k}})}-\widetilde{a}_\mathfrak{k}(v_{H,\mathfrak{k}},v_{H,\mathfrak{k}})\right]\\
         & = [ \widetilde{a}_\mathfrak{k}(v_{H,\mathfrak{k}},v_{H,\mathfrak{k}}) -\lambda]+\left[    \frac{\widetilde{a}_\mathfrak{k}(v_{H,\mathfrak{k}},v_{H,\mathfrak{k}})}{\widetilde{m}_\mathfrak{k}(v_{H,\mathfrak{k}},v_{H,\mathfrak{k}})}\norm{v^\mathrm{ms}_{H,\mathfrak{k}}}^2_{L^2(\Omega)}-   \frac{\widetilde{a}_\mathfrak{k}(v_{H,\mathfrak{k}},v_{H,\mathfrak{k}})}{\widetilde{m}_\mathfrak{k}(v_{H,\mathfrak{k}},v_{H,\mathfrak{k}})}\widetilde{m}_\mathfrak{k}(v_{H,\mathfrak{k}},v_{H,\mathfrak{k}})\right]\\
         &= [ \widetilde{a}_\mathfrak{k}(v_{H,\mathfrak{k}},v_{H,\mathfrak{k}}) -\lambda]+ \frac{\widetilde{a}_\mathfrak{k}(v_{H,\mathfrak{k}},v_{H,\mathfrak{k}})}{\widetilde{m}_\mathfrak{k}(v_{H,\mathfrak{k}},v_{H,\mathfrak{k}})}\big[ {m}_\mathfrak{k}(v_{H,\mathfrak{k}},v_{H,\mathfrak{k}})-\widetilde{m}_\mathfrak{k}(v_{H,\mathfrak{k}},v_{H,\mathfrak{k}})\big]\\
           &= a(v^\mathrm{ms}_{H,\mathfrak{k}},v_{H,\mathfrak{k}} ) -\lambda+ \frac{\widetilde{a}_\mathfrak{k}(v_{H,\mathfrak{k}},v_{H,\mathfrak{k}})}{\widetilde{m}_\mathfrak{k}(v_{H,\mathfrak{k}},v_{H,\mathfrak{k}})}\big[ {m}_\mathfrak{k}(v_{H,\mathfrak{k}},v_{H,\mathfrak{k}})-\widetilde{m}_\mathfrak{k}(v_{H,\mathfrak{k}},v_{H,\mathfrak{k}})\big]\\
         &\leq \norm{u-v^\mathrm{ms}_{H,\mathfrak{k}}}^2_{\mathcal{A}}-\lambda\norm{u-v^\mathrm{ms}_{H,\mathfrak{k}}}^2_{L^2(\Omega)} + \norm{(\mathscr{C}-\mathscr{C}_\mathfrak{k})v_{H,\mathfrak{k}}}_{\mathcal{A}}\norm{v_{H,\mathfrak{k}}}_{\mathcal{A}}\\
         & \hspace{2cm}+ \frac{\widetilde{a}_\mathfrak{k}(v_{H,\mathfrak{k}},v_{H,\mathfrak{k}})}{\widetilde{m}_\mathfrak{k}(v_{H,\mathfrak{k}},v_{H,\mathfrak{k}})}\left[{m}_\mathfrak{k}(v_{H,\mathfrak{k}},v_{H,\mathfrak{k}})-\widetilde{m}_\mathfrak{k}(v_{H,\mathfrak{k}},v_{H,\mathfrak{k}})\right].\qedhere
    \end{split}
\end{equation*}
\end{proof}

\subsection{Main results for the MLOD method}
\label{sec: 4.2}
In this section, we present the main estimates for MLOD eigenfunctions and eigenvalues.  

\begin{theorem}[Eigenfunction estimates]
\label{EQT:765}
Let $\lambda^{(l)}$ be the $l$-th eigenvalue of the continuous problem \eqref{eqnn: 2.7}  with multiplicity $r$, i.e., $\lambda^{(l)}=\lambda^{(l+1)}= \cdots = \lambda^{(l+r-1)}$, and let the corresponding eigenspace $E^{(l+j)} \subset \mathcal{H}^1_\#(\Omega),$ for ${j=0,1, \cdots, r-1}$ be spanned by the $r$-many orthonormal basis functions. Let $\left\{\tilde{u}_{H,\mathfrak{k}}^{(l+j),\mathrm{ms}} \right\}_{j=0}^{r-1} \subset V^{\mathrm{ms}}_{H,\mathfrak{k}}$  be an orthonormal set of multiscale eigenfunctions approximating the MLOD problem in \eqref{eqnn: 5.399} with $\norm{\tilde{u}_{H,\mathfrak{k}}^{(l+j),\mathrm{ms}}}_{L^2(\Omega)}=1\; \text{ for every } j$. Assume $l+r-1 \leq N$, where $N=\mathrm{dim}(V_H)$. Then, for each  $j= 0, 1, \cdots, r-1$, there exists a function $\Tilde{u}^{(l+j)}\in E^{(l+j)}$ with $\norm{\Tilde{u}^{(l+j)}}_{L^2(\Omega)}=1$ such that
\begin{equation*}
\norm{\Tilde{u}^{(l+j)}-\tilde{u}_{H,\mathfrak{k}}^{(l+j),\mathrm{ms}}}_{\mathcal{A}} \lesssim H^2 + g(\mathfrak{k}).
\end{equation*}
The hidden constant depends only on the spectral bounds, domain $\Omega$, multiplicity $r$ and the corresponding continuous eigenvalue $\lambda^{(l)}$.
\end{theorem}

We emphasize that approximation order $H^2$ is the same as for the Galerkin LOD method (with or without localization) presented in \cite{MR3343928}, cf. \eqref{est:LOD}. Note that the localization parameter $\mathfrak k$ can be chosen sufficiently large to obtain an overall convergence order $H^2$.

We further note that the choice of $\Tilde{u}^{(l+j)} \in E^{(l+j)}$ such that $\norm{\Tilde{u}^{(l+j)}-\tilde{u}_{H,\mathfrak{k}}^{(l+j),\mathrm{ms}}}_{\mathcal{A}} \to 0$ as $H \to 0$ and $\mathfrak{k} \to \infty$ is made a posteriori, i.e., for each computed eigenfunction, a corresponding function in the exact eigenspace is chosen to minimize the energy norm error. Such matching is guaranteed by classical results for perturbations of compact operators as discussed in \cite{MR716134,MR1115240}. This standard approach allows us to make individual per-function estimates even in the presence of eigenvalue multiplicities, since the continuous and discrete eigenspaces converge as subspaces under mesh refinement and localization.

\begin{proof}
For any $j=0,1,\cdots, r-1$, with $\Tilde{u}^{(l+j)}\in E^{(l+j)}$ we can write
\begin{equation*}
     \norm{\Tilde{u}^{(l+j)}-\tilde{u}_{H,\mathfrak{k}}^{(l+j),\mathrm{ms}}}_{\mathcal{A}} 
         \leq  \norm{\Tilde{u}^{(l+j)}-{u}_{H,\mathfrak{k}}^{(l+j),\mathrm{ms}}}_{\mathcal{A}} + \norm{{u}_{H,\mathfrak{k}}^{(l+j),\mathrm{ms}}-\tilde{u}_{H,\mathfrak{k}}^{(l+j),\mathrm{ms}}}_{\mathcal{A}}
\end{equation*}
with ${u}_{H,\mathfrak{k}}^{(l+j)} \in V^{\text{ms}}_H $ being the approximate solution to the LOD problem in \eqref{eqn: LOD}. Since $V^{\text{ms}}_{H,\mathfrak{k}} \in \mathcal{H}^1_\#$ and $\Tilde{u}^{(l+j)}\in \mathcal{H}^1_\#$, all terms can be interpreted within the same space. We can then proceed as follows using the triangle inequality
\begin{equation*}
    \begin{split}
    \label{eqnn563}
        \norm{\Tilde{u}^{(l+j)}-\tilde{u}_{H,\mathfrak{k}}^{(l+j),\mathrm{ms}}}_{\mathcal{A}} 
         & \leq  \norm{\Tilde{u}^{(l+j)}-{u}_{H,\mathfrak{k}}^{(l+j),\mathrm{ms}}}_{\mathcal{A}} + \norm{{u}_{H,\mathfrak{k}}^{(l+j)}-\tilde{u}_{H,\mathfrak{k}}^{(l+j)}}_{\mathcal{A}}+\norm{ \mathscr{C}_\mathfrak{k}\left({u}_{H,\mathfrak{k}}^{(l+j)}-\tilde{u}_{H,\mathfrak{k}}^{(l+j)}\right) }_{\mathcal{A}}\\
         & \leq \norm{\Tilde{u}^{(l+j)}-{u}_{H,\mathfrak{k}}^{(l+j),\mathrm{ms}}}_{\mathcal{A}} +2\norm{{u}_{H,\mathfrak{k}}^{(l+j)}-\tilde{u}_{H,\mathfrak{k}}^{(l+j)}}_{\mathcal{A}}+ \norm{(\mathscr{C} - \mathscr{C}_\mathfrak{k})\left({u}_{H,\mathfrak{k}}^{(l+j)}-\tilde{u}_{H,\mathfrak{k}}^{(l+j)}\right)}_{\mathcal{A}}\\
         & \lesssim  H^2 + \left(g(\mathfrak{k})\right)^{1/2}H + \left(2+g(\mathfrak{k})\right)\norm{{u}_{H,\mathfrak{k}}^{(l+j)}-\tilde{u}_{H,\mathfrak{k}}^{(l+j)}}_{\mathcal{A}},
    \end{split}
\end{equation*}
where the second last step follows from the stability of $\mathscr{C}$ in the energy norm. In the last step we used the estimates from \eqref{est:LOD} for the $l$-th eigenvalue with multiplicity $r$. These estimates for $r>1$ are justified by the analysis in \cite{MR3343928}. Furthermore, the function $\Tilde{u}^{(l+j)}$ is chosen to minimize the energy norm error with respect to the MLOD eigenfunction $\Tilde{u}^{(l+j),\text{ms}}_{H,\mathfrak{k}}$. Therefore, it is not necessarily the best matching function for the LOD approximation ${u}^{(l+j),\text{ms}}_{H,\mathfrak{k}}$. However, since all approximate eigenspaces converge to $E^{(l+j)}$ in norm, the estimates used in \eqref{est:LOD} are valid for estimating the first term in the last step above, and the convergence rate is preserved up to a constant. We have further used Theorem \ref{thm:5.2.4} to approximate $ \norm{(\mathscr{C} - \mathscr{C}_\mathfrak{k})\left({u}_{H,\mathfrak{k}}^{(l+j)}-\tilde{u}_{H,\mathfrak{k}}^{(l+j)}\right)}_{\mathcal{A}}$ to arrive at the last step. What is left to bound is $\norm{{u}_{H,\mathfrak{k}}^{(l+j)}-\tilde{u}_{H,\mathfrak{k}}^{(l+j)}}_{\mathcal{A}}$, which we can estimate by using compact operators. Suppose in Lemma \ref{L: 6.2.1}, $M1$ is given by the LOD solving \eqref{eqn: LOD} and $M2$ is given by the MLOD solving \eqref{eqnn: 5.399}. Then, with $a_1={a}_\mathfrak{k}$, $a_2 = \widetilde{a}_\mathfrak{k}$, $m_1={m}_\mathfrak{k}$, $m_2=\widetilde{m}_\mathfrak{k}$, $T^1_H={T}^\mathfrak{k}_H$ and $T^2_H=\widetilde{T}^\mathfrak{k}_H$, for the $(l+j)$-th eigenfunction it yields
\begin{equation*}
\begin{split}
    \norm{{T}^\mathfrak{k}_Hf^{(l+j)}-\widetilde{T}^\mathfrak{k}_{H}f^{(l+j)}}^2_{\mathcal{A}}   &\lesssim    ({m}_\mathfrak{k}-\widetilde{m}_\mathfrak{k})(f^{(l+j)},{T}^\mathfrak{k}_Hf^{(l+j)}-\widetilde{T}^\mathfrak{k}_Hf^{(l+j)})\\
        &\hspace{2.1cm} +(\widetilde{a}_\mathfrak{k}-{a}_\mathfrak{k})(\widetilde{T}^\mathfrak{k}_Hf^{(l+j)}, {T}^\mathfrak{k}_Hf^{(l+j)}-\widetilde{T}^\mathfrak{k}_Hf^{(l+j)}).
\end{split}
\end{equation*}
Since $\widetilde{T}^\mathfrak{k}_Hf^{(l+j)}$ solves for \eqref{eqn: LOD}, by Lemma \ref{Lem: 7.75} and Lemma \ref{Lem: 764}, we obtain that
\begin{equation*}
\begin{split}
 \norm{{T}^\mathfrak{k}_Hf^{(l+j)}-\widetilde{T}^\mathfrak{k}_{H}f^{(l+j)}}_{\mathcal{A}}  & \lesssim  \norm{f^{(l+j)}}_\mathcal{A}H^2  + \norm{\widetilde{T}^\mathfrak{k}_Hf^{(l+j)}}_{\mathcal{A}} g(\mathfrak{k}).
\end{split}
\end{equation*}

Since ${T}^\mathfrak{k}_Hf^{(l+j)}$, $\widetilde{T}^\mathfrak{k}_Hf^{(l+j)} \in V_H$, there exist some ${v}_0$, $\tilde{v}_0 \in V_H$ such that 
  ${T}^\mathfrak{k}_Hf^{(l+j)}={v}_0$ and $\widetilde{T}^\mathfrak{k}_Hf^{(l+j)}=\tilde{v}_0$ solving \eqref{E: 5.68} and \eqref{E: 5.69} for LOD and MLOD respectively. On the other hand, \eqref{eqn: LOD} and \eqref{eqnn: 5.399} solve for eigensolutions on $V_H$ for the LOD and MLOD respectively. Moreover, the pair \eqref{E: 5.68} and \eqref{eqn: LOD}  (similarly \eqref{E: 5.69} and \eqref{eqnn: 5.399}) produce the same eigenfunctions and reciprocal eigenvalues to the LOD and MLOD problems. Therefore, for $f^{(l+j)} \in E^{(l+j)}$ with   $\norm{f^{(l+j)}}_{L^2(\Omega)}=1$, the analysis of compact operators for approximating eigenfunctions presented in \cite{MR1035176} yields that
\begin{equation}
\begin{split}
\norm{{u}^{(l+j)}_{H,\mathfrak{k}}-\tilde{u}^{(l+j)}_{H,\mathfrak{k}}}_\mathcal{A}  &\lesssim \left(\lambda^{(l+j)}\right)^{1/2}H^2 + \norm{\tilde{u}^{(l+j)}_{H,\mathfrak{k}}}_{\mathcal{A}}g(\mathfrak{k}).
\label{eqnn:774}
\end{split}
\end{equation}

\noindent We are left to bound $\norm{\tilde{u}^{(l+j)}_{H,\mathfrak{k}}}_{\mathcal{A}}$.
For $\mathfrak{k}$ large enough, by the coercivity of $\widetilde{a}(\cdot, \cdot)$ shown in \cite{MR3422449}, for any $v\in \mathcal{H}^1_\#(\Omega)$
\begin{equation*}
    \norm{v}_{\mathcal{A}}^2 \lesssim  \widetilde{a}_\mathfrak{k}(v,v).
\end{equation*}
If we set $v = \Tilde{u}_{H,\mathfrak{k}}^{(l+j)} $, we get that
\begin{equation}
\begin{split}
 \label{eq: 7.74}  \norm{\Tilde{u}_{H,\mathfrak{k}}^{(l+j)}}_{\mathcal{A}}^2 \lesssim \;  \abs{\widetilde{a}_\mathfrak{k}\left(\Tilde{u}_{H,\mathfrak{k}}^{(l+j)},\Tilde{u}_{H,\mathfrak{k}}^{(l+j)}\right)} &=  \; \abs{\tilde{\lambda}_{H,\mathfrak{k}}^{(l+j)}\widetilde{m}_\mathfrak{k}\left(\Tilde{u}_{H,\mathfrak{k}}^{(l+j)},\Tilde{u}_{H,\mathfrak{k}}^{(l+j)}\right)} \\ & \leq  \widetilde{C}_{\lambda} \abs{\lambda^{(l+j)}}\abs{m\left(\Tilde{u}_{H,\mathfrak{k}}^{(l+j)},\Tilde{u}_{H,\mathfrak{k}}^{(l+j)}\right)}=\widetilde{C}_{\lambda}\abs{ \lambda^{(l+j)}}\norm{\Tilde{u}_{H,\mathfrak{k}}^{(l+j)}}^2_{L^2(\Omega)}.
\end{split}
\end{equation}
The third step follows from the bound $\abs{\tilde{\lambda}^{(l+j)}_{H,\mathfrak k}} \leq c \abs{\lambda^{(l+j)}},$ for some $c\in \mathbb{R}$ by Babu\v{s}ka-Osborn theory. This is justified by Lemma \ref{L: 6.2.1} and the known analysis of the source problem (see above), which in particular yield the convergence $\tilde{\lambda}^{(l+j)}_{H,\mathfrak k}$ to $\lambda^{(l+j)}$ as $H \to 0$ and $\mathfrak{k} \to \infty$.
Since $ \norm{\Tilde{u}_{H,\mathfrak{k}}^{(l+j)}}_{L^2(\Omega)} = \norm{I_H\tilde{u}_{H,\mathfrak{k}}^{(l+j),\text{ms}}}_{L^2(\Omega)} \leq \widetilde{C}\norm{\tilde{u}_{H,\mathfrak{k}}^{(l+j),\text{ms}}}_{L^2(\Omega)}=\widetilde{C}$ by definition, $\norm{\Tilde{u}_{H,\mathfrak{k}}^{(l+j)}}_{\mathcal{A}} $ is clearly bounded above by a positive constant.
Therefore, with an underlying constant that depends on the continuous eigenvalue $\lambda^{(l)}$, it holds that
\begin{equation*}
\begin{split}
     \norm{\Tilde{u}^{(l+j)}-\tilde{u}_{H,\mathfrak{k}}^{(l+j),\rm{ms}}}_{\mathcal{A}} &\lesssim 2H^2 +2g(\mathfrak{k}) + H^2g(\mathfrak{k})+ g^2(\mathfrak{k})  + \left(g(\mathfrak{k})\right)^{1/2}H + H^2.
\end{split}
\end{equation*}
The result easily follows by only considering the dominating terms.
 \end{proof}
 
We now turn to the eigenvalue estimate. Typically, one expects the square of the eigenfunction error, but here the perturbation terms need to be considered as well. Consequently, we obtain a convergence order of $H^2$ also for the eigenvalues and stress that this reduced order in comparison to \cite{MR3343928} stems from the use of the FE mass matrix, cf. Lemma \ref{Lem: 7.75}.

\begin{theorem}[Eigenvalue estimates] Let $\lambda^{(l)}$ be the $l$-th eigenvalue of the continuous problem \eqref{eqnn: 2.7} with multiplicity $r$, i.e., $\lambda^{(l)}=\lambda^{(l+1)}= \cdots = \lambda^{(l+r-1)}$. Let $\tilde{\lambda}^{(l+j)}_{H,\mathfrak{k}}$ denote the $(l+j)$-th eigenvalue of the discrete multiscale problem \eqref{eqnn: 5.399} for $j=0,1 \cdots, r-1$, and assume that $\norm{\tilde{u}_{H,\mathfrak{k}}^{(l+j),\mathrm{ms}}}_{L^2(\Omega)}=1$ for every $j$. Then, for each $j = 0, \cdots, r-1$
\begin{equation*}
\begin{split}
\abs{\frac{\tilde{\lambda}^{(l+j)}_{H,\mathfrak{k}}-\lambda^{(l+j)}}{\lambda^{(l+j)}} }   \lesssim g(\mathfrak{k}) + H^2.
    \end{split}
\end{equation*}
\end{theorem}

\begin{proof}
Note that $ \frac{\widetilde{a}_\mathfrak{k}\left(\tilde{u}^{(l+j)}_{H,\mathfrak{k}},\tilde{u}^{(l+j)}_{H,\mathfrak{k}}\right)}{\widetilde{m}_\mathfrak{k}\left(\tilde{u}^{(l+j)}_{H,\mathfrak{k}},\tilde{u}^{(l+j)}_{H,\mathfrak{k}}\right)} =\tilde{\lambda}^{(l+j)}_{H,\mathfrak{k}} $. Also remind that $\lambda^{(l+j)}$ of \eqref{eqnn: 2.7} is positive. We use Lemma \ref{Lem: 7.4.4} with $v_{H,\mathfrak{k}} = \tilde{u}_{H,\mathfrak{k}}^{(l+j)}$ to obtain that
\begin{equation*}
    \begin{split}
       \tilde{\lambda}^{(l+j)}_{H,\mathfrak{k}}-\lambda^{(l+j)} &= \frac{\widetilde{a}_\mathfrak{k}\left(\tilde{u}^{(l+j)}_{H,\mathfrak{k}},\tilde{u}^{(l+j)}_{H,\mathfrak{k}}\right)}{\widetilde{m}_\mathfrak{k}\left(\tilde{u}^{(l+j)}_{H,\mathfrak{k}},\tilde{u}^{(l+j)}_{H,\mathfrak{k}}\right)}-\lambda^{(l+j)}\\
      & \leq  \norm{u^{(l+j)}-\tilde{u}^{(l+j),\mathrm{ms}}_{H,\mathfrak{k}}}^2_{\mathcal{A}}-\lambda^{(l+j)} \norm{u^{(l+j)}-\tilde{u}^{(l+j),\mathrm{ms}}_{H,\mathfrak{k}}}^2_{L^2(\Omega)} +\norm{\left(\mathscr{C}-\mathscr{C}_\mathfrak{k}\right)\tilde{u}^{(l+j)}_{H,\mathfrak{k}}}_{\mathcal{A}}\norm{\tilde{u}^{(l+j)}_{H,\mathfrak{k}}}_{\mathcal{A}} \\
      & \hspace{4cm} + \tilde{\lambda}^{(l+j)}_{H,\mathfrak{k}}\left({m}_\mathfrak{k}\left(\tilde{u}^{(l+j)}_{H,\mathfrak{k}},\tilde{u}^{(l+j)}_{H,\mathfrak{k}}\right)-\widetilde{m}_\mathfrak{k}\left(\tilde{u}^{(l+j)}_{H,\mathfrak{k}},\tilde{u}^{(l+j)}_{H,\mathfrak{k}}\right)\right).
  \end{split}
\end{equation*}

Since  $\abs{\tilde{\lambda}^{(l+j)}_{H,\mathfrak{k}}}$ can be bounded above by the absolute continuous eigenvalue, by Theorem \ref{EQT:765}, Theorem \ref{thm:5.2.4} and Lemma \ref{Lem: 7.75}, we derive that
           \begin{equation}
            \begin{split} 
            \label{eqn: eig_MLOD}
   \abs{\frac{\tilde{\lambda}^{(l+j)}_{H,\mathfrak{k}}-\lambda^{(l+j)}}{\lambda^{(l+j)}} }   & \lesssim \left(H^2 + g(\mathfrak{k})\right)^2 +  g(\mathfrak{k})\norm{\tilde{u}^{(l+j)}_{H,\mathfrak{k}}}_{\mathcal{A}}^2 + H^2.
    \end{split}
\end{equation}
The assertion follows since $\norm{\tilde{u}^{(l+j)}_{H,\mathfrak{k}}}_{\mathcal{A}}^2$ can be bounded above as previously shown in \eqref{eq: 7.74}. 
\end{proof}

\begin{remark}
\noindent Given that $g(\mathfrak{k})\approx H^2$, if we compare equation \eqref{eqn: eig_MLOD} with the eigenvalue results form \eqref{est:LOD}, we can clearly notice that the order of the first term in \eqref{eqn: eig_MLOD} yields the same order as in \eqref{est:LOD} whereas the reduction of order from $H^4$ to $H^2$ follows as a consequence of the error due to the localization corrector term (second term in \eqref{eqn: eig_MLOD}) in the Petrov-Galerkin method regardless of the choice of the mass matrix. Therefore, by choosing the FEM mass matrix instead of a LOD based mass matrix, we obtain an efficient solver at the same convergence rate $H^2$.
\end{remark}

\subsection{OLOD method}
We now present our analytical results for the eigenfunctions and eigenvalues in the OLOD method. As explained in earlier sections, we utilize the MLOD results as an intermediate step between the continuous and OLOD solutions. Recall that moving from the MLOD to the OLOD method, only the stiffness matrix on the left-hand side changes. Analyzing this additional consistency error with the techniques already used for the MLOD method, we show the same error rates now also for the OLOD method.

\begin{theorem}[Eigenfunction estimates]
\label{thm: 7.7}
Suppose the OLOD problem in \eqref{Eq:5.103} is well-posed. 
Let $\lambda^{(l)}$ be the $l$-th eigenvalue of the continuous problem \eqref{eqnn: 2.7} with multiplicity $r$, i.e., $\lambda^{(l)}=\lambda^{(l+1)}= \cdots = \lambda^{(l+r-1)}$, and let the corresponding eigenspace $E^{(l+j)} \subset \mathcal{H}^1_\#(\Omega)$ for ${j=0,1, \cdots, r-1}$ be spanned by the $r$-many orthonormal basis functions. Let $\left\{\hat{u}_{H,\mathfrak{k}}^{(l+j),\mathrm{ms}} \right\}_{j=0}^{r-1} \subset V^\mathrm{ms}_{H, \mathfrak k}$ be an orthonormal set of multiscale eigenfunctions approximating the OLOD problem in \eqref{Eq:5.103}, where $\hat{u}^{(l+j), \mathrm{ms}}_{H,\mathfrak{k}}= \hat{u}^{(l+j)}_{H,\mathfrak{k}}-\widehat{\mathscr{C}}_\mathfrak{k}\hat{u}^{(l+j)}_H$, $\hat{u}^{(l+j)}_{H,\mathfrak{k}}\in V_H$. Assume $l+r-1 \leq N$, where $N=\mathrm{dim}(V_H)$. Then, for each  $j= 0, 1, \cdots, r-1$, there exists a function $\hat{u}^{(l+j)}\in E^{(l+j)}$ with $\norm{\hat{u}^{(l+j)}}_{L^2(\Omega)}=1$ such that\begin{equation*}
\label{EQ:7.109}
     \norm{\hat{u}^{(l+j)}-\hat{u}^{(l+j), \mathrm{ms}}_{H,\mathfrak{k}}}_\mathcal{A} \lesssim H^2 + g(\mathfrak{k}) +  \left(\max_{T\in \mathcal{T}_H}E_T\right)\mathfrak{k}^{d/2}.
\end{equation*}
The hidden constant depends only on the spectral bounds, domain $\Omega$, multiplicity $r$ and the corresponding continuous eigenvalue $\lambda^{(l)}$.
\end{theorem}
In the following proof, we use similar arguments used in Theorem \ref{EQT:765} for the validity of the traingle inequality and the estimates used for bounding $\norm{\hat{u}^{(l+j)}-\tilde{u}^{(l+j),\mathrm{ms}}_{H, \mathfrak{k}}}_\mathcal{A}$.
\begin{proof}
For $j=0,1,\cdots, r-1$, we have that
\begin{equation*}
    \begin{split}
        \norm{\hat{u}^{(l+j)}-\hat{u}^{(l+j),\mathrm{ms}}_{H,\mathfrak{k}}}_\mathcal{A}
        &\leq \norm{\hat{u}^{(l+j)}-\tilde{u}^{(l+j),\mathrm{ms}}_{H, \mathfrak{k}}}_\mathcal{A} +\norm{\tilde{u}^{(l+j)}_{H,\mathfrak{k}}-{\mathscr{C}_\mathfrak{k}}\tilde{u}^{(l+j)}_{H\mathfrak{k}}-\hat{u}^{(l+j)}_{H,\mathfrak{k}} + \widehat{\mathscr{C}}_\mathfrak{k}\hat{u}^{(l+j)}_{H,\mathfrak{k}}}_\mathcal{A}\\
        &\leq \norm{\hat{u}^{(l+j)}-\tilde{u}^{(l+j),\mathrm{ms}}_{H, \mathfrak{k}}}_\mathcal{A} +\norm{\tilde{u}^{(l+j)}_{H,\mathfrak{k}}-\hat{u}^{(l+j)}_{H,\mathfrak{k}}}_\mathcal{A}\\
        & \hspace{3.05cm} +\norm{\widehat{\mathscr{C}}_\mathfrak{k}\hat{u}^{(l+j)}_{H,\mathfrak{k}}+{\mathscr{C}_\mathfrak{k}}\hat{u}^{(l+j)}_{H,\mathfrak{k}}-{\mathscr{C}_\mathfrak{k}}\hat{u}^{(l+j)}_{H,\mathfrak{k}} - {\mathscr{C}}_\mathfrak{k}\tilde{u}^{(l+j)}_{H,\mathfrak{k}}}_\mathcal{A} \\
         & \leq \norm{\hat{u}^{(l+j)}-\tilde{u}^{(l+j),\mathrm{ms}}_{H, \mathfrak{k}}}_\mathcal{A} +\norm{\tilde{u}^{(l+j)}_{H,\mathfrak{k}}-\hat{u}^{(l+j)}_{H,\mathfrak{k}}}_\mathcal{A} \\
        & \hspace{3.05cm} +\norm{{\mathscr{C}_\mathfrak{k}}\left(\tilde{u}^{(l+j)}_{H,\mathfrak{k}}-\hat{u}^{(l+j)}_{H,\mathfrak{k}}\right)}_\mathcal{A}+\norm{\left(\widehat{\mathscr{C}}_\mathfrak{k}-{\mathscr{C}}_\mathfrak{k}\right)\hat{u}^{(l+j)}_{H,\mathfrak{k}}}_\mathcal{A}\\
       &  \lesssim \norm{\hat{u}^{(l+j)}- \tilde{u}^{(l+j),\mathrm{ms}}_{H, \mathfrak{k}}}_\mathcal{A} +2\norm{\tilde{u}^{(l+j)}_{H,\mathfrak{k}}-\hat{u}^{(l+j)}_{H,\mathfrak{k}}}_\mathcal{A}+\norm{\left(\widehat{\mathscr{C}}_\mathfrak{k}-{\mathscr{C}}_\mathfrak{k}\right)\hat{u}^{(l+j)}_{H,\mathfrak{k}}}_\mathcal{A}.
    \end{split}
\end{equation*}

\noindent We already have estimates for all the terms above except the second. This can be easily found by applying Lemma
\ref{L: 6.2.1} with $M1$ given by the OLOD and $M2$ by the MLOD solving \eqref{Eq:5.103} and \eqref{eqnn: 5.399} respectively. Then, with $a_1 = \widehat{a}_\mathfrak{k}, a_2 =\widetilde{a}_\mathfrak{k}, m_1 = \widehat{m}_\mathfrak{k} = \widetilde{m}_\mathfrak{k} = m_2, T^1_H = \widehat{T}^\mathfrak{k}_H$ and $T^2_H = \widetilde{T}^\mathfrak{k}_H$ together with \eqref{EQN: 5.108} yields
\begin{equation*}
    \begin{split}
         \norm{\widetilde{T}^{\mathfrak{k}}_Hf^{(l+j)}-\widehat{T}^{\mathfrak{k}}_Hf^{(l+j)}}_{\mathcal{A}}\lesssim  &  \; \left(\max\limits_{T\in \mathcal{T}_H}E_T\right)\mathfrak{k}^{d/2} \norm{\widehat{T}^\mathfrak{k}_Hf^{(l+j)}}_\mathcal{A}.
    \end{split}
\end{equation*}
From the coercivity condition in \eqref{equ:7} and bounds similar to \eqref{eq: 7.74}, together with the interpolation property we derive a constant upper bound for $\norm{\hat{u}^{(l+j)}_{H,\mathfrak{k}}}_\mathcal{A}$. Hence, by similar arguments from compact operator theory used to obtain \eqref{eqnn:774} we yield 
\begin{equation}
\begin{split}
\label{eq: 7114}
\norm{\tilde{u}^{(l+j)}_{H,\mathfrak{k}}-\hat{u}^{(l+j)}_{H,\mathfrak{k}}}_\mathcal{A} &\lesssim \left(\max\limits_{T\in \mathcal{T}_H}E_T\right)\mathfrak{k}^{d/2}.
\end{split}
\end{equation}
Theorem \ref{EQT:765}, \eqref{eq:7109} of Theorem \ref{Thm:prelim} and \eqref{eq: 7114} imply that
\begin{equation*}
\begin{split}
     \norm{\hat{u}^{(l+j)}-\hat{u}^{(l+j), \rm{ms}}_{H,\mathfrak{k}}}_\mathcal{A} \lesssim  H^2 + g(\mathfrak{k}) +2\left(\max\limits_{T\in \mathcal{T}_H}E_T \right)\mathfrak{k}^{d/2} + \left(\max_{T\in \mathcal{T}_H}E_T\right)\mathfrak{k}^{d/2}.\qedhere
\end{split}
\end{equation*}
\end{proof}

\begin{theorem}[Eigenvalue estimate]
\label{EigThm: Off}Let $\lambda^{(l)}$ be the $l$-th eigenvalue of the continuous problem \eqref{eqnn: 2.7}  with multiplicity $r$, i.e., $\lambda^{(l)}=\lambda^{(l+1)}= \cdots = \lambda^{(l+r-1)}$. Let $\hat{\lambda}^{(l+j)}_{H,\mathfrak{k}}$ denote the $(l+j)$-th eigenvalue of the discrete multiscale problem given by \eqref{Eq:5.103} for $j = 0, \cdots, r-1$. Consider the same setting as in Theorem \ref{thm: 7.7} where $\norm{\hat{u}^{(l+j),\rm{ms}}_{H,\mathfrak{k}}}_{L^2(\Omega)}=1$ for every $j$. Then, for each $j = 0, \cdots, r-1$
\begin{equation*}
\begin{split}
\abs{\frac{\hat{\lambda}^{(l+j)}_{H,\mathfrak{k}}-\lambda^{(l+j)}}{\lambda^{(l+j)}} } \lesssim \; & H^2+g(\mathfrak{k})+\left(\max\limits_{T\in \mathcal{T}_H}E_T\right)\mathfrak{k}^{d/2} + \left(\max_{T\in \mathcal{T}_H}E_T\right)^2\mathfrak{k}^{d}.
\end{split}
\end{equation*}
\end{theorem}

\begin{proof}
Note that $ \frac{\widehat{a}_\mathfrak{k}\left(\hat{u}^{(l+j)}_{H,\mathfrak{k}},\hat{u}^{(l+j)}_{H,\mathfrak{k}}\right)}{\widehat{m}_\mathfrak{k}\left(\hat{u}^{(l+j)}_{H,\mathfrak{k}},\hat{u}^{(l+j)}_{H,\mathfrak{k}}\right)} =\hat{\lambda}^{(l+j)}_{H,\mathfrak{k}} $. Then, we obtain that
\begin{align}
       \hat{\lambda}^{(l+j)}_{H,\mathfrak{k}}-\lambda^{(l+j)} &= \frac{\widehat{a}_\mathfrak{k}\left(\hat{u}^{(l+j)}_{H,\mathfrak{k}},\hat{u}^{(l+j)}_{H,\mathfrak{k}}\right)}{\widehat{m}_\mathfrak{k}\left(\hat{u}^{(l+j)}_{H,\mathfrak{k}},\hat{u}^{(l+j)}_{H,\mathfrak{k}}\right)}-\lambda^{(l+j)}\NNN
      & \leq \widehat{a}_\mathfrak{k}\left(\hat{u}^{(l+j)}_{H,\mathfrak{k}},\hat{u}^{(l+j)}_{H,\mathfrak{k}}\right) -\lambda^{(l+j)} +\frac{\widehat{a}_\mathfrak{k}\left(\hat{u}^{(l+j)}_{H,\mathfrak{k}},\hat{u}^{(l+j)}_{H,\mathfrak{k}}\right)}{\widehat{m}_\mathfrak{k}\left(\hat{u}^{(l+j)}_{H,\mathfrak{k}},\hat{u}^{(l+j)}_{H,\mathfrak{k}}\right)}- \widehat{a}_\mathfrak{k}\left(\hat{u}^{(l+j)}_{H,\mathfrak{k}},\hat{u}^{(l+j)}_{H,\mathfrak{k}}\right)\NNN 
      & = \left(\widehat{a}_\mathfrak{k}- \widetilde{a}_\mathfrak{k}\right)\left(\hat{u}^{(l+j)}_{H,\mathfrak{k}},\hat{u}^{(l+j)}_{H,\mathfrak{k}}\right) + \left( \widetilde{a}_\mathfrak{k}\left(\hat{u}^{(l+j)}_{H,\mathfrak{k}},\hat{u}^{(l+j)}_{H,\mathfrak{k}}\right) - \lambda^{(l+j)}\right)\NNN
      & \hspace{4.22cm} + \hat{\lambda}^{(l+j)}_{H,\mathfrak{k}}\left( \norm{\hat{u}^{(l+j),\mathrm{ms}}_{H,\mathfrak{k}}}^2_{L^2(\Omega)} - \widehat{m}_\mathfrak{k}\left(\hat{u}^{(l+j)}_{H,\mathfrak{k}},\hat{u}^{(l+j)}_{H,\mathfrak{k}} \right)\right)\NNN
      & \leq \abs{\left(\widehat{a}_\mathfrak{k}- \widetilde{a}_\mathfrak{k} \right)\left(\hat{u}^{(l+j)}_{H,\mathfrak{k}},\hat{u}^{(l+j)}_{H,\mathfrak{k}}\right)} + \abs{\widetilde{a}_\mathfrak{k}\left(\hat{u}^{(l+j)}_{H,\mathfrak{k}},\hat{u}^{(l+j)}_{H,\mathfrak{k}}\right) - \lambda^{(l+j)}}\NNN
      & \hspace{4.45cm} + \abs{\widehat{C}_{{\lambda}}\lambda^{(l+j)}}\,\abs{\norm{\hat{u}^{(l+j),\mathrm{ms}}_{H,\mathfrak{k}}}^2_{L^2(\Omega)} - \widehat{m}_\mathfrak{k}\left(\hat{u}^{(l+j)}_{H,\mathfrak{k}},\hat{u}^{(l+j)}_{H,\mathfrak{k}}\right )}
      \label{EQ:7126}
      \end{align}
\noindent  where the constant $\widehat{C}_{\lambda}$ follows from a similar explanation for $\widetilde{C}_\lambda$ in \eqref{eq: 7.74}. We use Lemma \ref{Lem: 5.1.4}, Theorem \ref{thm:5.2.4}, \eqref{eq:7109} of Theorem \ref{Thm:prelim} and the stability of $\mathscr{C}$ in the energy norm together with the triangle inequality to estimate the last term as

\begin{align}
\label{eqnn: 7.129}
&\!\!\!\!\!\!\!\!\!\!\!\! \norm{\hat{u}^{(l+j),\mathrm{ms}}_{H, \mathfrak{k}}}^2_{L^2(\Omega)} - \widehat{m}\left(\hat{u}^{(l+j)}_{H, \mathfrak{k}},\hat{u}^{(l+j)}_{H, \mathfrak{k}} \right)\NNN
\hspace{1cm}&= m\left(\left(1-\widehat{\mathscr{C}}_\mathfrak{k}\right)\hat{u}^{(l+j)}_{H,\mathfrak{k}}, \left(1-\widehat{\mathscr{C}}_\mathfrak{k}\right)\hat{u}^{(l+j)}_{H,\mathfrak{k}}\right) - m\left(\hat{u}^{(l+j)}_{H,\mathfrak{k}},\hat{u}^{(l+j)}_{H,\mathfrak{k}}\right)\NNN 
& \leq \abs{ m(\widehat{\mathscr{C}}_\mathfrak{k}\hat{u}^{(l+j)}_{H,\mathfrak{k}}, \widehat{\mathscr{C}}_\mathfrak{k}\hat{u}^{(l+j)}_{H,\mathfrak{k}})} + 2\abs{m(\widehat{\mathscr{C}}_\mathfrak{k}\hat{u}^{(l+j)}_{H,\mathfrak{k}}, \hat{u}^{(l+j)}_{H,\mathfrak{k}})}\NNN
&\lesssim \left(\norm{\widehat{\mathscr{C}}_\mathfrak{k}\hat{u}^{(l+j)}_{H,\mathfrak{k}}}^2_\mathcal{A}+2\norm{\widehat{\mathscr{C}}_\mathfrak{k}\hat{u}^{(l+j)}_{H, \mathfrak{k}}}_\mathcal{A}\norm{\hat{u}^{(l+j)}_{H, \mathfrak{k}}}_\mathcal{A}\right)H^2\NNN
&\lesssim \Bigg( \norm{\left(\mathscr{C}-{\mathscr{C}}_\mathfrak{k}\right)\hat{u}^{(l+j)}_{H,\mathfrak{k}}}^2_\mathcal{A}+\norm{\left(\mathscr{C}_\mathfrak{k}-\widehat{\mathscr{C}}_\mathfrak{k}\right)\hat{u}^{(l+j)}_{H,\mathfrak{k}}}^2_\mathcal{A} +2\norm{\left(\mathscr{C}-{\mathscr{C}}_\mathfrak{k}\right)\hat{u}^{(l+j)}_{H, \mathfrak{k}}}_\mathcal{A}\norm{\hat{u}^{(l+j)}_{H,\mathfrak{k}}}_\mathcal{A}\NNN
\displaybreak
& \hspace{2cm}+2\norm{\left(\mathscr{C}_\mathfrak{k}-\widehat{\mathscr{C}}_\mathfrak{k}\right)\hat{u}^{(l+j)}_{H,\mathfrak{k}}}_\mathcal{A}\norm{\hat{u}^{(l+j)}_{H,\mathfrak{k}}}_\mathcal{A} +3\norm{\hat{u}^{(l+j)}_{H,\mathfrak{k}}}^2_\mathcal{A}\Bigg)H^2\NNN
&\lesssim \norm{\hat{u}^{(l+j)}_{H,\mathfrak{k}}}^2_\mathcal{A}\left(g^2(\mathfrak{k}) + 2g(\mathfrak{k})  +3+ \left(\max_{T\in \mathcal{T}_H}E_T\right)^2 \mathfrak{k}^{d} + 2\left(\max_{T\in \mathcal{T}_H}E_T\right)\mathfrak{k}^{d/2}\right)H^2.
\end{align}

\noindent Next, let us find an estimate for the middle term in \eqref{EQ:7126}.
Recall that by definition, it holds that
\begin{equation*}
\label{EQ:bc}
  { \widetilde{a}_\mathfrak{k}\left(\hat{u}^{(l+j)}_{H,\mathfrak{k}},\hat{u}^{(l+j)}_{H,\mathfrak{k}}\right) - \lambda^{(l+j)}} = { a\left((1-{\mathscr{C}}_\mathfrak{k})\hat{u}^{(l+j)}_{H,\mathfrak{k}},\hat{u}^{(l+j)}_{H,\mathfrak{k}}\right) - \lambda^{(l+j)}}. 
\end{equation*}
Since $\left(1-{\mathscr{C}}_\mathfrak{k}\right)\hat{u}^{(l+j)}_{H,\mathfrak{k}}\in V^{\text{ms}}_{H,\mathfrak{k}}$, by the construction in \eqref{EQNN:784}, we obtain that
\begin{equation*}
    \begin{split}
   a\left(\left(1-{\mathscr{C}}_\mathfrak{k}\right)\hat{u}^{(l+j)}_{H,\mathfrak{k}},\hat{u}^{(l+j)}_{H,\mathfrak{k}}\right) & \leq     \norm{\left(1-\mathscr{C}_\mathfrak{k}\right)\hat{u}^{(l+j)}_{H,\mathfrak{k}}}^2_\mathcal{A}+  \norm{\left(\mathscr{C}-\mathscr{C}_\mathfrak{k}\right)\hat{u}^{(l+j)}_{H,\mathfrak{k}}}_{\mathcal{A}}\norm{\hat{u}^{(l+j)}_{H,\mathfrak{k}}}_{\mathcal{A}}\\
   & \leq \norm{\left(1-\widehat{\mathscr{C}}_\mathfrak{k}\right)\hat{u}^{(l+j)}_{H,\mathfrak{k}}}^2_\mathcal{A}+ \norm{\left(\mathscr{C}_\mathfrak{k}-\widehat{\mathscr{C}}_\mathfrak{k}\right)\hat{u}^{(l+j)}_{H,\mathfrak{k}}}^2_\mathcal{A}+  \norm{\left(\mathscr{C}-\mathscr{C}_\mathfrak{k}\right)\hat{u}^{(l+j)}_{H,\mathfrak{k}}}_{\mathcal{A}}\norm{\hat{u}^{(l+j)}_{H,\mathfrak{k}}}_{\mathcal{A}}.
    \end{split}
\end{equation*}

\noindent Noting that $\norm{\left(1-\widehat{\mathscr{C}}_\mathfrak{k}\right)\hat{u}^{(l+j)}_{H,\mathfrak{k}}}_\mathcal{A}^2=a\left(\left(1-\widehat{\mathscr{C}}_\mathfrak{k}\right)\hat{u}^{(l+j)}_{H,\mathfrak{k}},\left(1-\widehat{\mathscr{C}}_\mathfrak{k}\right)\hat{u}^{(l+j)}_{H,\mathfrak{k}}\right)$, by combining with \eqref{eqn: 777}, we arrive at
\begin{align}
     { \widetilde{a}_\mathfrak{k}\left(\hat{u}^{(l+j)}_{H,\mathfrak{k}},\hat{u}^{(l+j)}_{H,\mathfrak{k}}\right) - \lambda^{(l+j)}} &\leq    \norm{u^{(l+j)}-\hat{u}^{(l+j),\text{ms}}_{H,\mathfrak{k}
   }}^2_{\mathcal{A}}-\lambda^{(l)}\norm{u^{(l+j)}-\hat{u}^{(l+j),\text{ms}}_{H,\mathfrak{k}}}^2_{L^2(\Omega)} \NNN
   & \hspace{2cm} + \norm{\left(\mathscr{C}_\mathfrak{k}-\widehat{\mathscr{C}}_\mathfrak{k}\right)\hat{u}^{(l+j)}_{H,\mathfrak{k}}}^2_{\mathcal{A}} +  \norm{\left(\mathscr{C}-{\mathscr{C}}_\mathfrak{k}\right)\hat{u}^{(l+j)}_{H,\mathfrak{k}}}_{\mathcal{A}}\norm{\hat{u}^{(l+j)}_{H, \mathfrak{k}}}_{\mathcal{A}}\NNN
 &\lesssim  \norm{u^{(l+j)}-\hat{u}^{(l+j),\text{ms}}_{H,\mathfrak{k}
   }}^2_{\mathcal{A}} + \norm{\left(\mathscr{C}_\mathfrak{k}-\widehat{\mathscr{C}}_\mathfrak{k}\right)\hat{u}^{(l+j)}_{H,\mathfrak{k}}}^2_{\mathcal{A}} +  \norm{\left(\mathscr{C}-{\mathscr{C}}_\mathfrak{k}\right)\hat{u}^{(l+j)}_{H,\mathfrak{k}}}_{\mathcal{A}}\norm{\hat{u}^{(l+j)}_{H, \mathfrak{k}}}_{\mathcal{A}}.
   \label{EQ:7128}
\end{align}
The above inequality can be easily estimated using Theorems \ref{thm: 7.7}, \ref{thm:5.2.4} and \eqref{eq:7109} of Theorem \ref{Thm:prelim}.

\noindent Using \eqref{EQN: 5.108}, and substituting \eqref{eqnn: 7.129} and \eqref{EQ:7128} in \eqref{EQ:7126}, we find the desired estimate as
 
\begin{equation*}
    \begin{split}
    \left|\frac{\hat{\lambda}^{(l+j)}_{H,\mathfrak{k}}-\lambda^{(l+j)}}{\lambda^{(l+j)}} \right| &\lesssim \left(\max\limits_{T\in \mathcal{T}_H}E_T\right)\mathfrak{k}^{d/2}\norm{\hat{u}^{(l+j)}_{H}}_\mathcal{A}^2 + \left( H^2 + g(\mathfrak{k}) +  \left(\max_{T\in \mathcal{T}_H}E_T\right)\mathfrak{k}^{d/2}\right)^2  +  \left(\max_{T\in \mathcal{T}_H}E_T\right)^2\mathfrak{k}^{d}\norm{\hat{u}^{(l+j)}_{H,\mathfrak{k}}}^2_\mathcal{A}\\  & \hspace{2cm} + g(\mathfrak{k})\norm{\hat{u}^{(l+j)}_{H,\mathfrak{k}}}^2_\mathcal{A} + \left(3 + g(\mathfrak{k}) + \left(\max_{T\in \mathcal{T}_H}E_T\right)\mathfrak{k}^{d/2}\right)H^2\norm{\hat{u}^{(l+j)}_{H,\mathfrak{k}}}^2_\mathcal{A}\\
    & \lesssim \left(H^2+g(\mathfrak{k})+\left(\max\limits_{T\in \mathcal{T}_H}E_T \right)\mathfrak{k}^{d/2}+ \left(\max_{T\in \mathcal{T}_H}E_T\right)^2\mathfrak{k}^{d} \right)\norm{\hat{u}^{(l+j)}_{H, \mathfrak{k}}}_\mathcal{A}^2.
    \end{split}
\end{equation*}
\noindent The result follows, since $\norm{\hat{u}^{(l+j)}_{H, \mathfrak{k}}}_\mathcal{A}^2$, similar to \eqref{eq: 7.74}, can be bounded above by a constant due to $\widehat{a}(\cdot, \cdot)$-coercivity given in \eqref{equ:7}.
\end{proof}

\section{Numerical experiments}
\label{sec:5}
In this section, we illustrate our theoretical developments with several numerical experiments. We focus on the OLOD method and its alternate variant with the $s$-values from Section \ref{subsec:alternate}. For details in the calculation of $E_T$ (if desired for error control) and discussions of the run-time complexity, we refer to \cite[Section 5]{MR4378546}. Note that aspects on the implementation of the LOD (without offline-online strategy) are explained in \cite{MR3926249}. Our code is based on \cite{gridlodcode} and \cite{randompertubation}, and freely available at \cite{dilini_kolombage_2025_14243509}. 

\begin{remark}
In our numerical experiments, we will focus on the approximation of the lowest non-trivial eigenvalue, but emphasize that the theory of Section \ref{sec:4} is also valid for multiple and/or higher eigenvalues.
By our focus on the lowest eigenvalue, we circumvent any possible numerical challenges well-known for high eigenvalues such as spectral pollution \cite{MR1058692, MR2640293}.
Furthermore, it is known that (higher) eigenvalues may ``cross'' in its dependence on the random parameter, leading to difficulties in the correct identification and ``matching'' of corresponding eigenvalues and eigenfuctions \cite{MR1335452}. We emphasize that our method and error estimates are targeted at accurate and efficient approximations for single realizations or samples, and that we do not consider the possibility of eigenvalue crossings in our numerical experiments.
\end{remark}
In our experiments discussed in Sections \ref{sec: 5.1} and \ref{subsec:5.2}, we consider the domain $\Omega =[0,1]^d$ for $d=1,2$. The spectral bounds are set to $\alpha=0.1$ and $\beta=1$. We consider $\mathfrak{k}=3$ for the LOD formulation and use the standard finite element method on a fine mesh with the fixed mesh size  $h=2^{-8}$ as our reference solver. We further set the multiscale parameter $\varepsilon = 2^{-7}$. We measure the root mean square error (RMSE) between the reference eigenvalue and the OLOD eigenvalues against various probabilities $p$ and the coarse mesh size $H$ over $N=200$ samples. To compute eigenvalues, we use the built-in functions in the \textit{scipy.sparse} package, which are based on the ARPACK solver. In all cases we use a tolerance $10^{-4}$.

\begin{remark}
The elliptic eigenvalue problem may produce multiple eigenvalues. For reasons of comparison, we consider an averaged eigenvalue of the two smallest non-trivial eigenvalues in our experiment. 
\end{remark}

\subsection{Experiments for the OLOD method}
\label{sec: 5.1}

\subsubsection{One-dimensional experiments}

In the one-dimensional case, we randomly assign the value $0.1$ to each interval of length $\varepsilon$. Figure \ref{fig:image3} shows the error convergence with respect to $H$ (left) and the defect probability $p$ (right).  In Figure \ref{fig:subim3.2}, we observe a second order convergence, which confirms our theoretical findings in Theorem \ref{EigThm: Off}. For finer $H$, we see that the convergence becomes slower the larger $p$ is. This is most probably due to the consistency error in the offline-online strategy, which is expected to grow with $p$ and to dominate LOD spatial discretization error for sufficiently fine meshes. Similar $H$-convergence rates can also be observed in the two-dimensional experiments (data not shown). Having confirmed the expected $H$-convergence, we completely focus on the error dependence on $p$ in the following and choose rather small values of $H$ to concentrate on the consistency error.

Figure \ref{subim3.1} illustrates the RMSE against probability $p$ for various $H$ sizes. We notice a less than $2\%$ RMSE up to $10\%$ of defect probability. We further observe that on each $H$, the corresponding error is almost constant for probabilities smaller than $5\%$. Thereafter the error is only slightly growing.

\begin{figure}[h]
\begin{subfigure}{0.45\textwidth}
\includegraphics[width=0.9\linewidth, height=5cm]{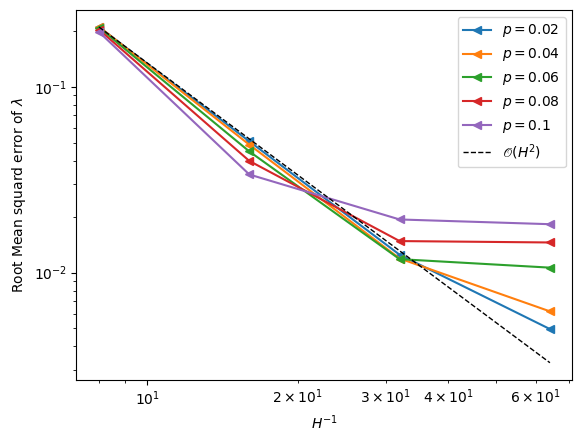}
\caption{$H$-convergence }
\label{fig:subim3.2}
\end{subfigure}
\begin{subfigure}{0.45\textwidth}
\includegraphics[width=0.9\linewidth, height=5cm]{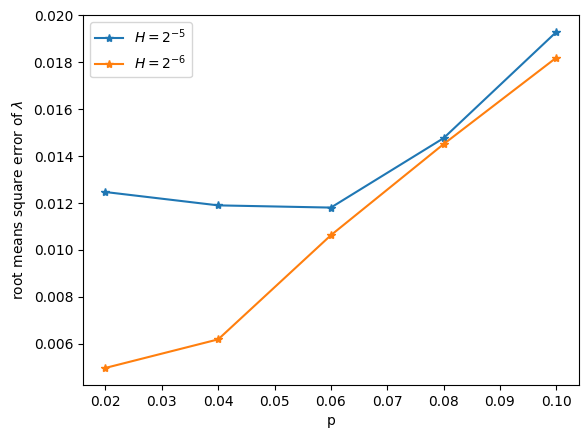} 
\caption{$p$-dependence}
\label{subim3.1}
\end{subfigure}
\caption{The one-dimensional random checkerboard.}
\label{fig:image3}
\end{figure}

\subsubsection{Two-dimensional experiments}
Figure \ref{fig:image4} illustrates two-dimensional results for the random checkerboard (Figure \ref{fig:1.1}) and random erasure (Figure \ref{fig:1.2}) in Section \ref{sec:Intro}.

Again, our OLOD method shows a very promising performance with a RMSE of a few percent for reasonable defect probabilities. For instance, the random checkerboard results in Figure \ref{fig:subim4.1} shows RMSEs less than $12\%$ for $H=2^{-5}$ and a defect probability up to $10\%$. In the random erasure, Figure \ref{fig:subim4.2}, RMSEs are even lower, more precisely for $H=2^{-5}$, we observe less than $2\%$ error for material defects up to $10\%$.
\begin{figure}[h]
\begin{subfigure}{0.45\textwidth}
\includegraphics[width=0.9\linewidth, height=5cm]{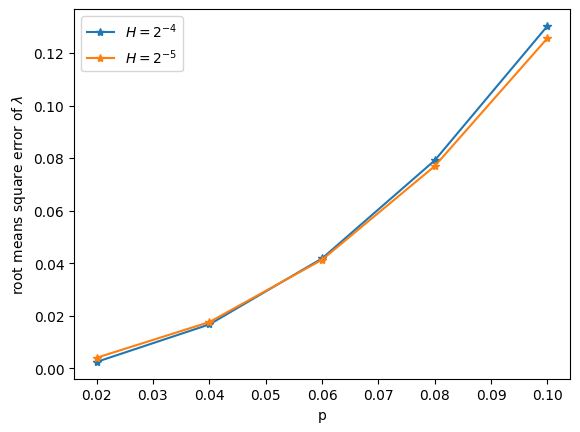}
\caption{Random checkerboard}
\label{fig:subim4.1}
\end{subfigure}
\begin{subfigure}{0.45\textwidth}
\includegraphics[width=0.9\linewidth, height=5cm]{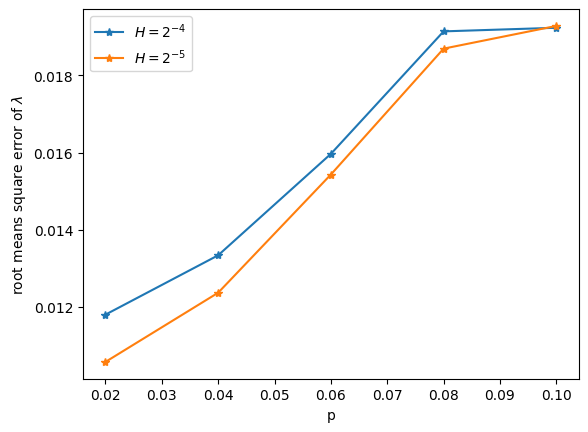}
\caption{Random erasure}
\label{fig:subim4.2}
\end{subfigure}
\caption{The OLOD (sum constraint one) results in two-dimensions.}
\label{fig:image4}
\end{figure}

\subsection{Experiments with the alternate offline-online strategy}
\label{subsec:5.2}
In this section, we compare the accuracy of the OLOD and its alternate offline-online strategy (from Section \ref{subsec:alternate}) for the random checkerboard. The one- and two-dimensional comparisons are depicted in the Figures \ref{fig:com1} and \ref{fig:com2}, respectively. The $H_{uty}$ and $H_s$ denote the $H$-discretization corresponding to the sum constraint one and the alternate methods respectively. We only compare the results for the smaller mesh sizes $H=2^{-5}$ and $H=2^{-6}$ to focus on the consistency error and reduce the influence of any discretization errors. In both dimensions the alternate method produces improved results for larger $p$ as expected. Particularly in two dimensions, we can observe an RMSE less than $\sim 6\%$ for about $10\%$ of material defects. For the one-dimensional case, as expected both methods produce better results with an RMSE of $\sim 4\%$ at a $10\%$ material error. However, for much smaller probabilities of material error in the $1$D case, slightly better results can be noticed for the OLOD with sum constraint one. 
We emphasize that this is no contradiction of our theoretical derivations as we only considered limits $\varepsilon \to 0$ there. As expected from the choice of $s$, the difference between the sum constraint one model and the alternate variant becomes more prominent with growing $p$. We emphasize that, although our derivations were based on one-dimensional heuristics, the two-dimensional results are very promising and underline the potential of our approach. We note that in some practical scenarios in the context of metamaterials, defect probabilities are typically kept below  $5\%$ (and sometimes even below $1\%$). Therefore, the defect probabilities considered in our experiments are already on the higher end and cover a broad range of realistic cases.

\begin{figure}[h]
\centering
\includegraphics[width=0.5\linewidth, height=6cm]{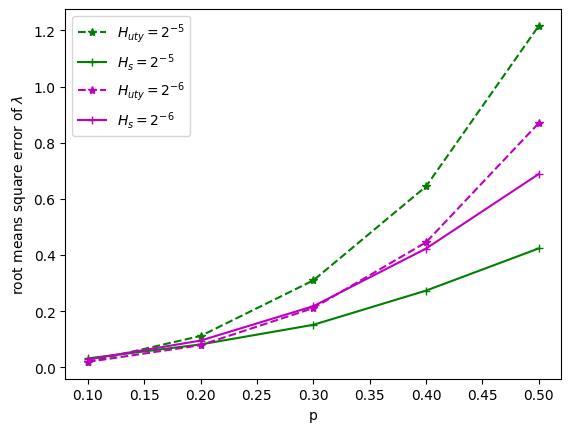}
\caption{Comparison of the offline-online strategy with sum constraint one and alternate version in one-dimension.}
\label{fig:com1}
\end{figure}

\begin{figure}[h]
\centering
\includegraphics[width=0.5\linewidth, height=6cm]{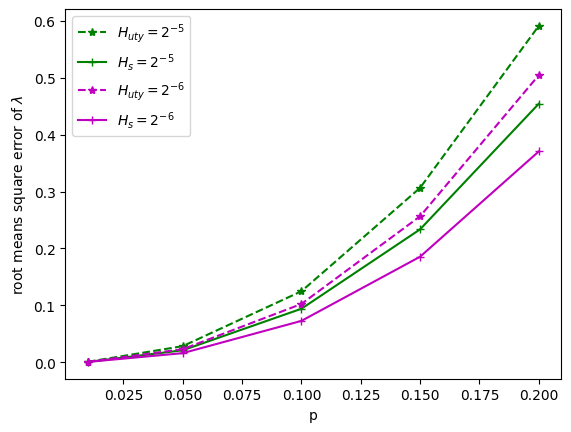}
\caption{Comparison of the offline-online strategy with sum constraint one and alternate version in two-dimension.}
\label{fig:com2}
\end{figure}

\section{Conclusion}
\label{sec:6}
We outlined a practically feasible approach of the Localized Orthogonal Decomposition (LOD) based offline-online strategy for the elliptic eigenvalue problem with randomly perturbed coefficients suited for Monte-Carlo-type simulations. We analyzed the errors in eigenvalue and eigenfunction approximation for any given realization of the coefficient. We further introduced a heuristic offline-online strategy that is uniquely adaptable to the defect distribution by exploiting the stochastic aspects of the problem. Our numerical experiments illustrate the good performance of our strategy, especially they indicate the superiority of the alternate method adapting the sum constraint. 
Since our derivation were based on rather simple one-dimensional heuristics, it might be interesting to investigate more sophisticated ``optimisation'' strategies as well as the application to other probability distributions in future research.
Moreover,  the periodicity assumption on the domain clearly plays a major role in this paper, which we hope to reduce in future research.

\section*{Acknowledgments}
This research is funded by the Deutsche Forschungsgemeinschaft (DFG, German Research Foundation) under project number 496556642. BV additionally acknowledges support by the Deutsche Forschungsgemeinschaft (DFG, German Research Foundation) under Germany's Excellence Strategy – EXC-2047/1 – 390685813.
The authors would like to thank the anonymous reviewers for their constructive feedback and insightful comments, which helped improve the clarity and quality of this work.
\bibliographystyle{alpha}
\bibliography{references}

\end{document}